\documentclass[square,numbers,sort&compress,3p]{elsarticle}

\usepackage{amsmath,amssymb,amsthm,bm}
\usepackage{booktabs,enumitem,setspace}
\usepackage{graphicx,float}
\usepackage{subcaption}
\usepackage{longtable}
\usepackage{array}

\usepackage[labelfont=bf,format=plain,
            justification=raggedright,
            singlelinecheck=false]{caption}

\newtheorem{theorem}{Theorem}[section]
\newtheorem{lemma}{Lemma}[section]
\newtheorem{corollary}{Corollary}[section]
\newtheorem{remark}{Remark}[section]

\usepackage[T1]{fontenc}
\usepackage[american]{babel}

\usepackage[colorlinks=true]{hyperref}
\usepackage{times}

\begin{document}

\begin{frontmatter}
\title{Sustainable Water Treatment through Fractional-Order Chemostat Modeling with Sliding Memory and Periodic Boundary Conditions: A Mathematical Framework for Clean Water and Sanitation}

\author[dept,ndrc]{Kareem T. Elgindy\corref{cor1}}
\ead{k.elgindy@ajman.ac.ae}
\address[dept]{Department of Mathematics and Sciences, College of Humanities and Sciences, Ajman University, P.O. Box 346, Ajman, United Arab Emirates}
\address[ndrc]{Nonlinear Dynamics Research Center (NDRC), Ajman University, P.O. Box 346, Ajman, United Arab Emirates}
\cortext[cor1]{Corresponding author}

\begin{abstract}
This work develops and analyzes a novel fractional-order chemostat system (FOCS) with a Caputo fractional derivative (CFD) featuring a sliding memory window and periodic boundary conditions (PBCs), designed to model microbial pollutant degradation in sustainable water treatment. By incorporating the Caputo fractional derivative with sliding memory (CFDS), the model captures time-dependent behaviors and memory effects in biological systems more realistically than classical integer-order formulations. We reduce the two-dimensional fractional differential equations (FDEs) governing substrate and biomass concentrations to a one-dimensional FDE by utilizing the PBCs. The existence and uniqueness of non-trivial, periodic solutions are established using the Carath\'{e}odory framework and fixed-point theorems, ensuring the system's well-posedness. We prove the positivity and boundedness of solutions, demonstrating that substrate concentrations remain within physically meaningful bounds and biomass concentrations stay strictly positive, with solution trajectories confined to a biologically feasible invariant set. Additionally, we analyze non-trivial equilibria under constant dilution rates and derive their stability properties. The rigorous mathematical results confirm the viability of FOCS models for representing memory-driven, periodic bioprocesses, offering a foundation for advanced water treatment strategies that align with Sustainable Development Goal 6 (Clean Water and Sanitation).
\end{abstract}
\begin{keyword}
Caputo fractional derivative \sep Carath\'{e}odory solution \sep Chemostat \sep Existence and uniqueness \sep Fourier-Gegenbauer \sep Periodic boundary conditions \sep Well-posedness.\\[0.5em]
\end{keyword}
\end{frontmatter}

\begin{longtable}{p{2.5cm} p{10cm}}
\caption{List of Acronyms} \\
\toprule
\textbf{Acronym} & \textbf{Definition} \\
\midrule
\endfirsthead
\toprule
\textbf{Acronym} & \textbf{Definition} \\
\midrule
\endhead
\bottomrule
\endfoot
a.e. & Almost Everywhere \\
CFD & Caputo Fractional Derivative \\
CFDS & Caputo Fractional Derivative with Sliding Memory \\
FD & Fractional Derivative \\
FDE & Fractional Differential Equation \\
FG-PS & Fourier-Gegenbauer pseudospectral \\
FOCM & Fractional-Order Chemostat Model \\
FOCS & Fractional-Order Chemostat System \\
PBC & Periodic Boundary Condition \\
PFOCP & Periodic Fractional Optimal Control Problem \\
SDG & Sustainable Development Goal \\
\end{longtable}

\section*{Introduction}
The chemostat, a bioreactor designed for the continuous cultivation of microorganisms, serves as a fundamental tool in sustainable water treatment applications such as pollutant degradation, where controlled microbial growth is critical. Classical chemostat models typically assume steady-state conditions, which simplify analysis but often fail to account for the dynamic and adaptive behaviors of microbial populations. Recent research has shown that periodic operation of the chemostat can improve performance by aligning with the time-varying nature of biological processes \citep{bayen2020improvement, elgindy2023new}. In particular, Elgindy \citep{elgindy2023new} demonstrated that optimal periodic control policies can significantly improve chemostat performance compared to steady-state operation, achieving up to 57\% reduction in time-averaged substrate concentration under certain conditions. This work attempts to bridge the gap between advanced mathematical modeling and practical water treatment challenges by introducing a fractional-order framework with sliding memory and PBCs. The findings build upon decades of theoretical development in chemostat dynamics, highlighting the potential for fractional calculus to revolutionize sustainable water treatment systems. 

The mathematical modeling of chemostat dynamics has evolved considerably since the pioneering work of Douglas and Rippin \citep{douglas1966unsteady}, who first observed that unsteady-state operation could outperform steady-state conditions in biochemical reactors. This phenomenon, known as ``overyielding,'' occurs when periodic variations in input parameters lead to improved system performance. Bayen et al. \citep{bayen2020improvement} further established that for chemostats exhibiting overyielding states, the optimal periodic control follows a bang-bang\footnote{A bang–bang function is a two-state, piecewise constant function that switches abruptly between two states \cite{elgindy2023new}.} strategy with even switching times. These findings highlight the importance of developing accurate mathematical frameworks that can capture the complex dynamics of periodically operated bioreactors. To address the limitations of classical integer-order models in capturing such complex behaviors, researchers have increasingly turned to fractional calculus as a more sophisticated mathematical approach.

Fractional calculus generalizes differentiation and integration to non-integer orders, providing a robust mathematical framework for modeling systems with long-term memory and non-local effects. This offers a significant improvement over classical integer-order models, enabling the incorporation of memory and hereditary effects in biological systems. The CFD, widely used in physical and biological systems, is particularly effective due to its compatibility with physically interpretable initial conditions. For a function $f(t)$ and fractional order $\alpha \in (0,1)$, the CFD is defined as:
\[
D^\alpha f(t) = \frac{1}{\Gamma(1-\alpha)} \int_0^t (t - \tau)^{-\alpha} f'(\tau) \, d\tau,
\]
where $\Gamma$ denotes the Gamma function. This integral formulation captures the system's historical evolution, weighting past states according to a power-law kernel. In the chemostat, this derivative models microbial growth and substrate dynamics that depend on both current and past states, such as the cumulative effects of nutrient availability or microbial adaptation. However, while the standard CFD offers theoretical advantages, its practical implementation in periodic systems presents computational challenges that require innovative solutions.

To improve the practicality of this approach, we introduce a fixed sliding memory effect into the Caputo framework, as proposed by \citet{bourafa2021periodic}, and modified later by \citet{elgindy2024fourierA,elgindy2024fourierB,elgindy2024fourierC}, primarily to address and improve its computational efficiency and accuracy, particularly when dealing with PFOCPs and fractional partial differential equations with periodic solutions. Unlike the standard CFD, which integrates over all past states from the initial time $t=0$, the fixed sliding memory approach restricts the memory window to a finite, time-dependent interval $[t - L, t]$, where $L > 0$ specifies the memory length. The CFDS is thus given by:
\begin{equation}\label{eq:caputo}
{}^{\text{MC}}_{L}D_t^\alpha f(t) = \frac{1}{\Gamma(1-\alpha)} \int_{t-L}^t (t - \tau)^{-\alpha} f'(\tau) \, d\tau.
\end{equation}
This modification offers several key advantages, particularly for modeling PFOCPs in chemostat systems. First, it preserves the periodicity of the FD for $T$-periodic functions, addressing a critical limitation of classical Gr\"{u}nwald-Letnikov, Riemann-Liouville, and CFDs, which do not maintain periodicity for non-constant, periodic functions \citep{elgindy2024fourierA}. This property is essential for accurately modeling chemostat dynamics under periodic operation, where microbial responses to cyclic nutrient inputs or environmental changes must remain consistent over each period. Second, the fixed sliding memory window improves computational efficiency by limiting the integration range to a finite interval, reducing the computational burden compared to evaluating integrals over the entire history from $t=0$. This is particularly beneficial for numerical simulations of complex bioreactor systems, where computational resources are often a limiting factor. Third, the sliding memory approach emphasizes recent dynamics, reflecting the biological reality that recent environmental conditions, such as substrate concentrations, typically have a greater impact on microbial behavior than distant past states. This makes the model more biologically relevant, as microbial populations in chemostats often exhibit finite memory effects, responding primarily to recent nutrient availability or environmental shifts. Finally, as demonstrated by Elgindy \citep{elgindy2024fourierA,elgindy2024fourierB}, the CFDS enables the transformation of PFOCPs into constrained NLPs, which can be efficiently solved using standard nonlinear programming problem solvers, facilitating practical implementation in chemostat control strategies. The theoretical foundation established by these fractional calculus advancements has inspired numerous researchers to explore their application in specific chemostat modeling scenarios.

Building on the advancements in fractional calculus, recent studies have further explored the application of FDs to chemostat modeling, improving the understanding of microbial dynamics under complex conditions. In their 2017 studies, \citet{zeinadini2017approximation,zeinadini2017numerical} extended the three-dimensional chemostat model by incorporating FDs and variable yield coefficients. The resulting nonlinear fractional-order systems were described by sets of FDEs utilizing the Gr\"{u}nwald-Letnikov FD, with state variables representing the nutrient concentration and the concentrations of two competing microorganisms. The growth kinetics of the microorganisms were modeled using the Monod model, also known as Michaelis-Menten dynamics, and the yield coefficients were expressed as functions of nutrient concentration involving terms of various powers. These studies primarily focused on the stability analysis and numerical approximation of the FOCM using nonstandard finite difference schemes. 

Four years later, \citet{aris2021dynamical} introduced a two-dimensional FOCM that further explored the integration of FDs and variable yield coefficients. This nonlinear fractional-order system was described by a set of FDEs employing the CFD, with state variables representing the substrate concentration and the cell mass concentration. The growth kinetics of the microorganisms were similarly modeled using the Monod expression, while the yield coefficient was formulated as a function of substrate concentration involving linear and constant terms. The study emphasized stability analysis and Hopf bifurcation of the FOCM, employing the Adams-type predictor-corrector method for numerical approximations. A year later, and further expanding on the investigation of FOCMs, \citet{mohd2022stability} analyzed the stability properties of another FOCM incorporating the CFD, a Monod growth function, a linear variable yield coefficient in the substrate concentration, and time delays to account for microbial response lag. The primary contribution was a numerical stability analysis of the model under varying fractional orders and time delays, focusing on equilibrium points and their stability, alongside numerical simulations. It was assumed that the system's mathematical formulation was well-defined for the chosen parameters and initial conditions. However, as with the earlier works, the existence, uniqueness, and well-posedness of the system's solutions were not explicitly proven. 

The aforementioned works have largely contributed to understanding FOCMs by highlighting the impact of FDs (and possibly time delays) on system stability. However, the lack of formal analysis regarding existence, uniqueness, and well-posedness limited their theoretical rigor. One of the early works to address these aspects within the fractional calculus framework dates back to the work of \citet{zhu2023forward}, who investigated a FOCM featuring non-monotonic, Haldane-type growth and random bounded disturbances on the input flow. This model extends classical chemostat models by incorporating memory effects through the CFD (order $\alpha \in (0,1]$) and stochasticity. The authors established the existence, uniqueness, positivity, and boundedness of a global solution using fixed-point theory, employing the Lipschitz continuity of nonlinear terms and properties of the Mittag-Leffler function.

A third and widely applied growth model is the Contois growth model, which describes the specific growth rate $\mu(s,x)$ as:
\begin{equation}\label{eq:contois}
\mu(s, x) = \frac{\mu_{\max} s}{K x + s},
\end{equation}
where $\mu_{\max} > 0$ is the maximum growth rate and $K > 0$ is the saturation constant. This model has demonstrated particular efficacy in modeling biomass growth for the optimization of chemostat performance and in various wastewater treatment applications \cite{bayen2020improvement, elgindy2023new}. Its success is partly attributable to its capacity to account for the inhibitory effects of elevated biomass concentrations, thereby rendering it well-suited for systems in which cell density significantly influences growth kinetics. Given the Contois model's ability to capture complex microbial interactions, it provides a suitable foundation for advancing chemostat modeling by incorporating fractional-order dynamics. These dynamics allow for a more accurate representation of memory-dependent behaviors in microbial ecosystems, which are critical for understanding periodic operations.

Building upon this robust theoretical foundation and addressing gaps in rigorous mathematical analysis, this study explores the theoretical properties of a FOCS, emphasizing the existence, uniqueness, and stability of periodic solutions under PBCs. By utilizing the CFDS and incorporating the Contois growth model, we develop a comprehensive mathematical framework that accurately captures the memory-dependent and time-varying dynamics of microbial ecosystems in wastewater treatment. This work advances the application of fractional-order modeling to biological systems, offering novel insights into the interplay between memory effects, periodic operation, and microbial growth kinetics governed by the Contois model in chemostat dynamics.

\subsection*{Notation and Preliminaries}
To improve clarity and accessibility, we consolidate the key mathematical notation and definitions used throughout this paper. This subsection provides a reference for readers, particularly those less familiar with fractional calculus or chemostat modeling.

\begin{itemize}
\item \textbf{Time and Periodicity}: $t \in [0, \infty)$ denotes time, and $T > 0$ is the fixed period for PBCs.
\item \textbf{State Variables}: $s(t) \in [0, s_{\text{in}}]$ represents the substrate (pollutant) concentration, and $x(t) > 0$ denotes the biomass concentration.
\item \textbf{System Parameters}: $s_{\text{in}} > 0$ is the inlet substrate concentration, $Y > 0$ is the yield coefficient, $D(t) \in [D_{\min}, D_{\max}]$ is the piecewise-continuous dilution rate where $D_{\min} > 0$, $\bar{D} = \frac{1}{T} \int_0^T D(t) \, dt$ is the average dilution rate, $\mu_{\max} > 0$ is the maximum growth rate, and $K > 0$ is the saturation constant in the Contois growth model. $\vartheta > 0$ is a characteristic time constant with units of time, introduced to ensure dimensional consistency in the fractional-order differential equations. 
\item \textbf{Fractional Calculus}: $\alpha \in (0, 1)$ is the fractional order, $\Gamma(\cdot)$ is the Gamma function, $L > 0$ is the sliding memory length, and $f'$ denotes the first derivative of the function $f(t)$. The CFDS is defined by \eqref{eq:caputo}. An extensive exposition on this operator and its applications can be found in \citet{bourafa2021periodic}, Elgindy \cite{elgindy2024fourierA,elgindy2024fourierB,elgindy2024fourierC}.
${}_L I^\alpha$ is the Riemann-Liouville fractional integral of order $\alpha$ over the interval $[t-L, t]$, defined as:
\begin{equation}\label{eq:June20250406}
{}_L I^\alpha f(t) = \frac{1}{\Gamma(\alpha)} \int_{t-L}^t (t - \tau)^{\alpha-1} f(\tau) \, d\tau.
\end{equation}
\item \textbf{Growth Kinetics}: The specific growth rate of the microorganisms, $\mu(s, x)$, follows the Contois model \eqref{eq:contois}.
\item \textbf{Function Spaces}: $AC_T$ denotes the space of absolutely continuous $T$-periodic functions with norm $\|s\|_{AC} = \|s\|_{\infty} + \|s'\|_{L^1}$, where $\|s\|_{\infty} = \sup_{t \in [0,T]} |s(t)|$ and $\|s'\|_{L^1} = \int_0^T |s'(t)|\,dt$. $X = \{s \in AC_T \mid 0 \leq s(t) \leq s_{\text{in}}\}$ is a compact and convex subset of $AC_T$ representing the set of feasible substrate concentrations.
\item \textbf{Operators}: $\Phi_s$ is the fixed-point operator defined for proving existence of periodic solutions.
\end{itemize}

These definitions ensure consistency and facilitate understanding of the mathematical framework developed in subsequent sections.

The remainder of this paper is organized as follows. Section \ref{sec:MD1} presents a detailed description of the FOCM with sliding memory and PBCs, establishing the mathematical framework that governs substrate and biomass concentrations. In Section \ref{sec:RFOCS1}, we reduce the two-dimensional FDEs to a one-dimensional FDE, simplifying the analysis while preserving the essential dynamics of the system. Within this section, Section \ref{subsec:UTS1} establishes the uniqueness of the trivial solution, providing a critical foundation for our subsequent analysis, while Section \ref{subsec:RODS1} derives the reduced one-dimensional system that enables more tractable mathematical treatment. Section \ref{subsec:NTEFODS} analyzes the non-trivial equilibria of the fractional one-dimensional system under constant dilution rates, deriving stability conditions that characterize the long-term behavior of the chemostat. Section \ref{subsec:WPFOCS} establishes the well-posedness of the system, with Section \ref{subsec:EPCS1} proving the existence of periodic Carath\'{e}odory solutions and Section \ref{subsec:UPCS1} demonstrating their uniqueness using the Carath\'{e}odory framework and fixed-point theorems. Section \ref{subsec:PosBound} focuses on the positivity and boundedness properties of solutions, demonstrating that the model maintains biologically meaningful concentrations within a physically feasible invariant set. In Section \ref{sec:numerical_simulations}, we present numerical simulations to validate the theoretical results, focusing on the existence, uniqueness, positivity, boundedness, and stability of periodic solutions for the FOCS with sliding memory and PBCs, using the FG-PS method. Finally, Section \ref{sec:Conc} concludes the paper with a discussion of the key contributions to mathematical theory and sustainable water treatment practices, and suggests directions for future research in this interdisciplinary field.

\section{Model Description}
\label{sec:MD1}
We consider a FOCS designed for sustainable water treatment, modeling the dynamics of substrate concentration $s(t)$ and biomass concentration $x(t)$ for continuous biological water treatment. The system operates over a fixed period $T$ with an average dilution rate $\bar{D}$, which corresponds to the total treated volume $\bar{Q}$ normalized by the chemostat volume $V$ and period $T$, i.e., $\bar{D} = \bar{Q}/(V T)$. This constraint ensures a consistent treatment capacity. The system is governed by the following FDEs with a sliding memory window:
\begin{align}
{}^{\text{MC}}_{L}D_t^\alpha s(t) &= \vartheta^{1-\alpha} \left[ -\frac{1}{Y} \mu(s(t), x(t))x(t) + D(t)(s_{\text{in}} - s(t)) \right],\label{eq:sysdyn1}\\
{}^{\text{MC}}_{L}D_t^\alpha x(t) &= \vartheta^{1-\alpha} [\mu(s(t), x(t)) - D(t)]\,x(t),\label{eq:sysdyn2}
\end{align}
with PBCs:
\begin{align}
s(t) &= s(t+T),\quad \forall t \in [0, \infty),\label{eq:periodic1}\\
x(t) &= x(t+T),\quad \forall t \in [0, \infty),\label{eq:periodic2}\\
D(t) &= D(t+T),\quad \forall t \in [0, \infty).\label{eq:periodic3}
\end{align}
The biomass growth model is assumed to be the Contois growth model \eqref{eq:contois}. The fractional dynamics are modeled using the CFDS defined by \eqref{eq:caputo}. The sliding memory window $[t-L, t]$ captures finite memory effects to improve the realism of microbial growth and substrate degradation modeling compared to integer-order models. To ensure dimensional consistency in the fractional-order model, we introduce the parameter $\vartheta > 0$, a characteristic time constant with units of time. Physically, $\vartheta$ can represent a reference time scale associated with the system's dynamics, quantifying the influence of past states within the sliding memory window $[t-L, t]$ to capture delayed microbial responses or cumulative nutrient effects critical for pollutant degradation. We assume that the substrate and biomass dynamics share the same characteristic time scale, which is reasonable given their coupled nature in the chemostat. The substrate is consumed by the biomass, and their dynamics are driven by the same microbial growth kinetics and dilution rate. A single $\vartheta$ reflects a unified memory effect, where the system's historical states influence both variables similarly, which is consistent with the chemostat's operation as a single bioreactor. In practice, $\vartheta$ can be tuned to reflect operational time scales, such as hydraulic retention time, to improve the model's applicability to real-world water treatment systems. When $\alpha = 1, \vartheta^{1-\alpha} = 1$, and we recover the classical integer-order chemostat model.

The system's periodic operation and sliding memory effects are key to optimizing pollutant degradation and ensuring consistent water quality, addressing critical needs in clean water and sanitation. This study focuses on the well-posedness of the system, the existence and uniqueness of periodic solutions, the analysis of non-trivial equilibria, their stability properties, and the positivity and boundedness of solutions, validated through numerical simulations, with applications to sustainable water treatment.

With the FOCM thoroughly defined, the subsequent analysis shifts towards reducing the multi-dimensional system into a more mathematically tractable form, thereby facilitating a rigorous analytical investigation.

\section{Reduction of FOCS}
\label{sec:RFOCS1}
To simplify the analysis, we reduce the two-dimensional FDE system \eqref{eq:sysdyn1}--\eqref{eq:sysdyn2} to a one-dimensional equation using the PBCs and properties of the CFDS. To this end, we apply the transformation from \cite{bayen2020improvement} that relates the substrate and biomass concentrations:
\begin{equation}\label{eq:trans}
z(t) = Y (s_{\text{in}} - s(t)) - x(t).
\end{equation}
The FD of $z(t)$ is:
\begin{equation}\label{eq:trans_deriv}
{}^{\text{MC}}_{L}D_t^\alpha z(t) = -Y {}^{\text{MC}}_{L}D_t^\alpha s(t) - {}^{\text{MC}}_{L}D_t^\alpha x(t).
\end{equation}
Substituting \eqref{eq:sysdyn1} and \eqref{eq:sysdyn2} into \eqref{eq:trans_deriv} yields:
\begin{align}
{}^{\text{MC}}_{L}D_t^\alpha z(t) &= \vartheta^{1-\alpha} \left[ -Y \left( -\frac{1}{Y} \mu(s(t), x(t)) x(t) + D(t) (s_{\text{in}} - s(t)) \right) - \left[\mu(s(t), x(t)) - D(t)\right] x(t) \right] \notag \\
&= \vartheta^{1-\alpha} D(t) \left[ -Y (s_{\text{in}} - s(t)) + x(t) \right] = -\vartheta^{1-\alpha} D(t)z(t).\label{eq:FDEnewK1}
\end{align}
Given the PBCs \eqref{eq:periodic1}-\eqref{eq:periodic3}, it follows that $z(t) = z(t+T)$. We now prove that the trivial solution $z(t) \equiv 0$ is the unique periodic solution.

\subsection{Uniqueness of the Trivial Solution}
\label{subsec:UTS1}
Multiply both sides of the FDE \eqref{eq:FDEnewK1} by $z(t)$ and integrate over $[0, T]$:
\begin{equation}\label{eq:energy}
\int_0^T z(t)\,{}^{\text{MC}}_{L}D_t^\alpha z(t) \, dt = -\vartheta^{1-\alpha} \int_0^T D(t) z^2(t) \, dt.
\end{equation}
This equation resembles a fractional energy-like dissipation relation. The left-hand side represents a nonlocal, memory-dependent energy loss over one cycle, influenced by the FD’s sliding memory window $[t-L, t]$. The right-hand side, with $D(t) > 0$ (since the dilution rate is positive in the chemostat), acts as a dissipation term proportional to $z^2(t)$. Because $D(t)$ is positive, the right-hand side is non-positive and strictly negative unless $z(t) = 0$ a.e. For periodic $z(t)$, the energy-like quantity cannot decrease indefinitely over successive cycles; thus, the only consistent solution is the trivial solution $z(t) = 0$. This confirms that no non-trivial periodic solutions exist, and hence $z(t) \equiv 0$. This implies that 
\begin{equation}\label{eq:biomass}
x(t) = Y (s_{\text{in}} - s(t)).
\end{equation}

\subsection{Reduced One-Dimensional System}
\label{subsec:RODS1}
Substituting \eqref{eq:biomass} into \eqref{eq:sysdyn1} gives:
\begin{align}
{}^{\text{MC}}_{L}D_t^\alpha s(t) &= \vartheta^{1-\alpha} \left[-\frac{1}{Y} \mu(s(t), Y (s_{\text{in}} - s(t))) Y (s_{\text{in}} - s(t)) + D(t) (s_{\text{in}} - s(t))\right] \notag \\
&= \vartheta^{1-\alpha} \left[-\mu(s(t), Y (s_{\text{in}} - s(t))) (s_{\text{in}} - s(t)) + D(t) (s_{\text{in}} - s(t))\right].
\end{align}
Define
\begin{equation}\label{eq:nu}
\nu(s(t)) = \mu(s(t), Y (s_{\text{in}} - s(t))) = \frac{\mu_{\max} s}{K Y (s_{\text{in}} - s) + s},
\end{equation}
where $\nu(s)$ encapsulates the growth kinetics under the constraint \eqref{eq:biomass}. The reduced system becomes
\begin{equation}\label{eq:1DFDS}
{}^{\text{MC}}_{L}D_t^\alpha s(t) = \vartheta^{1-\alpha} [D(t) - \nu(s(t))] (s_{\text{in}} - s(t)),
\end{equation}
with the periodic condition \eqref{eq:periodic1}. This one-dimensional FDE governs the substrate concentration $s(t)$ under a periodic dilution rate $D(t)$. Clearly, it admits the trivial solution $s(t) = s_{\text{in}}$, which corresponds to the washout state characterized by the condition $(s(t), x(t)) = (s_{\text{in}}, 0)$.

\subsection{Non-Trivial Equilibrium of the Fractional One-Dimensional System}
\label{subsec:NTEFODS}
To derive the non-trivial equilibria\footnote{The non-trivial equilibrium is the steady-state solution with $\bar{s} < s_{\text{in}}$, ensuring $x = Y(s_{\text{in}} - \bar{s}) > 0$.} of the reduced one-dimensional system \eqref{eq:1DFDS} with PBC \eqref{eq:periodic1}, we consider the steady-state condition where the substrate concentration $s$ is constant, i.e., $s(t) = \bar{s}$, and the dilution rate is constant, i.e., $D(t) = \bar{D}$. At steady state, the CFDS vanishes because ${}^{\text{MC}}_{L}D_t^\alpha s(t) = 0$ for a constant $s$. Thus, the steady-state condition for Eq. \eqref{eq:1DFDS} is given by
\begin{equation}
0 = [\bar{D} - \nu(\bar{s})] (s_{\text{in}} - \bar{s}).
\end{equation}
Since the biomass concentration $x(t) > 0$, then $\bar{s} < s_{\text{in}}$ by Eq. \eqref{eq:biomass}. For non-trivial equilibria to exist, we require:
\begin{equation}\label{eq:NTE1}
\nu(\bar{s}) = \frac{\mu_{\max} \bar{s}}{K Y (s_{\text{in}} - \bar{s}) + \bar{s}} = \bar{D}.
\end{equation}
The unique solution to this equation is given by
\begin{equation}\label{eq:NTE2}
\bar{s} = \frac{\bar{D} K Y s_{\text{in}}}{\bar{D} K Y + \mu_{\max} - \bar{D}},
\end{equation}
provided $\bar{D} < \mu_{\max}$, which guarantees a positive and physically meaningful equilibrium ($\bar{s} < s_{\text{in}}$). Thus, there exists exactly one non-trivial equilibrium for a given $\bar{D} < \mu_{\max}$.

While this section established equilibria under constant dilution rates, practical systems often operate under time-varying $D$. We now generalize these results to periodic, non-steady dilution rates by proving well-posedness.

\subsection{Well-Posedness of the FOCS}
\label{subsec:WPFOCS}
In this section, we aim to prove the existence of non-trivial, periodic solutions $s$ for the FDE \eqref{eq:1DFDS} satisfying the PBC \eqref{eq:periodic1} under a dilution rate $D$ satisfying the PBC \eqref{eq:periodic3}. To this end, let us denote the right-hand side of the FDE \eqref{eq:1DFDS} by $f(t, s(t))$:
\[f(t, s(t)) = \vartheta^{1-\alpha} [D(t) - \nu(s(t))] (s_{\text{in}} - s(t)).\]
The following lemma establishes some key regularity properties of $f$.

\begin{lemma}[Regularity of $f$]\label{lem:f_properties}
Under the given assumptions of the FOCM, the function $f$ satisfies the following two properties:
\begin{enumerate}[label=(\roman*)]
    \item For each fixed $s \in [0, s_{\text{in}}]$, $f$ is Lebesgue measurable in $t$.
    \item $f$ is Lipschitz continuous with respect to $s$.
\end{enumerate}
\end{lemma}
\begin{proof}
\begin{enumerate}[label=(\roman*)]
    \item Since $D$ is Lebesgue measurable by assumption, and $\nu$ and $(s_{\text{in}} - s)$ are continuous in $s$, their product with $D$ and any linear combinations thereof are Lebesgue measurable in $t$, for each fixed $s \in [0, s_{\text{in}}]$. 
    \item Since 
    \begin{equation}\label{eq:nuprime1}
    \nu'(s) = \frac{K Y \mu_{\max} s_{\text{in}}}{(K Y (s_{\text{in}} - s) + s)^2},    
    \end{equation}
     is continuous and thus bounded on the compact domain $[0, s_{\text{in}}]$, there exists $\nu_{\max} > 0$ such that:
    \[
    |\nu(s_2) - \nu(s_1)| \le \nu_{\max} |s_2 - s_1|,\quad \text{for all }s_1, s_2 \in [0, s_{\text{in}}], 
    \]
    by the Mean Value Theorem. Moreover, $\nu'(s)$ is inversely proportional to the square of the linear function $K Y (s_{\text{in}}-s)+s$, which achieves its minimum at the boundaries of the compact interval $[0, s_{\text{in}}]$. Thus, the maximum of $\nu'(s)$ occurs at the endpoints $s = 0$ or $s = s_{\text{in}}$:
\[
\nu_{\max} = \max \left\{ \nu'(0), \nu'(s_{\text{in}}) \right\} = \frac{\mu_{\max}}{s_{\text{in}}} \max \left\{ K Y, \frac{1}{K Y} \right\}.
\]
This shows that $\nu$ is Lipschitz continuous on $[0, s_{\text{in}}]$ with Lipschitz constant $\nu_{\max}$. On the other hand, observe that
 \begin{gather*}
    f(t, s_1) - f(t, s_2) = \vartheta^{1-\alpha} \left[[D(t) - \nu(s_1)](s_{\text{in}} - s_1) - [D(t) - \nu(s_2)](s_{\text{in}} - s_2)\right]\\
    = \vartheta^{1-\alpha} \left[D(t)(s_2 - s_1) + s_{\text{in}} [\nu(s_2) - \nu(s_1)] + \nu(s_1) s_1 - \nu(s_2) s_2\right].     
 \end{gather*}
    Add and subtract $\nu(s_1) s_2$ to obtain:
 \begin{gather*}
    |f(t, s_1) - f(t, s_2)| = \vartheta^{1-\alpha} \left[\left| D(t)(s_2 - s_1) + s_{\text{in}} [\nu(s_2) - \nu(s_1)] + \nu(s_1)(s_1 - s_2) + s_2 [\nu(s_1) - \nu(s_2)] \right|\right]\\
    \le \vartheta^{1-\alpha} \left[|D(t)| |s_2 - s_1| + s_{\text{in}} |\nu(s_2) - \nu(s_1)| + |\nu(s_1)| |s_1 - s_2| + |s_2| |\nu(s_1) - \nu(s_2)|\right],
\end{gather*}
by the triangle inequality. Since $D(t) \in [D_{\min}, D_{\max}]$, $|\nu(s)| \le \mu_{\max}$, $s_2 \le s_{\text{in}}$, then by using the Lipschitz condition for $\nu(s)$ we obtain:
    \[
    |f(t, s_1) - f(t, s_2)| \le \vartheta^{1-\alpha} (D_{\max} + \mu_{\max} + 2 s_{\text{in}} \nu_{\max})\,|s_2 - s_1|.
    \]
    Thus, $f$ is Lipschitz continuous in $s$ with Lipschitz constant $L_f = \vartheta^{1-\alpha} (D_{\max}  + \mu_{\max} + 2 s_{\text{in}} \nu_{\max})$.
\end{enumerate}
\end{proof}

As a consequence, we have the following corollary.

\begin{corollary}[Additional Regularity Properties of $f$]\label{cor:f_properties}
Let the assumptions of Lemma \ref{lem:f_properties} hold true. Then:
\begin{enumerate}[label=(\roman*)]
    \item $f$ is continuous in $s$, for each fixed $t \in [0, T]$.
    \item $f$ is bounded, for bounded $s \in [0, s_{\text{in}}]$.
\end{enumerate}
\end{corollary}
\begin{proof}
\begin{enumerate}[label=(\roman*)]
    \item Since $f$ is Lipschitz continuous in $s$ by Lemma \ref{lem:f_properties}, it is also continuous in $s$. 
    \item Since $f$ is continuous in $s$ by (i), and $[0, s_{\text{in}}]$ is a compact set, $f$ is bounded on $[0, s_{\text{in}}]$. 
\end{enumerate}
\end{proof}

To continue our proof, we need to apply the Carath\'{e}odory framework for FDEs to account for the piecewise continuous nature of the dilution rate input $D$. To this end, we need first to establish the integral form of the FDE governing the substrate concentration in the FOCS. The following lemma demonstrates the equivalence between the FDE and its corresponding Volterra-type integral equation using the inverse operator property of the CFDS. 

\begin{lemma}
Let $s: [0,\infty) \to [0,\infty)$ be an absolutely continuous function on every compact interval. Then, $s$ satisfies the FDE:
\begin{equation}\label{eq:EDT1}
{}_L^{\text{MC}} D_t^\alpha s(t) = f(t, s(t)),
\end{equation}
if and only if it satisfies the integral equation:
\begin{equation}\label{eq:Volterra1}
s(t) = s(t-L+k T) + \frac{1}{\Gamma(\alpha)} \int_{t-L}^t (t - \tau)^{\alpha-1} f(\tau, s(\tau)) \, d\tau,
\end{equation}
where $k$ is the smallest integer such that $t-L+k T \in [0, T]\,\forall t \ge 0$.
\end{lemma}
\begin{proof}
To prove the lemma, we use the inverse operator property of the CFD with a variable lower limit to show that the FDE is equivalent to the integral equation \eqref{eq:Volterra1}. Assume $s$ satisfies the FDE \eqref{eq:EDT1}, and apply the Riemann-Liouville fractional integral of order $\alpha$ over the interval $[t-L, t]$, defined by \eqref{eq:June20250406}, to both sides of the FDE to obtain:
\[
{}_L I^\alpha \left( {}_L^{\text{MC}} D_t^\alpha s(t) \right) = {}_L I^\alpha f(t, s(t)).
\]
For the CFD with a variable lower limit $t-L$, the inverse operator property states that applying the fractional integral ${}_L I^\alpha$ to ${}_L^{\text{MC}} D_t^\alpha s(t)$ recovers the function $s$ up to its value at the lower bound of the integration interval \cite{diethelm2010analysis}:
\[
{}_L I^\alpha \left( {}_L^{\text{MC}} D_t^\alpha s(t) \right) = s(t) - s(t-L),
\]
provided that $s$ is absolutely continuous, from which Eq. \eqref{eq:Volterra1} is derived with periodicity adjustment. Now, assume that $s$ satisfies Eq. \eqref{eq:Volterra1}. Since ${}_L^{\text{MC}} D_t^\alpha \left( {}_L I^\alpha g(t) \right) = g(t)$ for a sufficiently regular function $g$ (e.g., Lebesgue integrable), then By Lemma \ref{lem:f_properties} and Corollary \ref{cor:f_properties}:
\[
{}_L^{\text{MC}} D_t^\alpha \left( \frac{1}{\Gamma(\alpha)} \int_{t-L}^t (t - \tau)^{\alpha-1} f(\tau, s(\tau)) \, d\tau \right) = f(t, s(t)).
\]
Thus, the FDE is recovered, confirming equivalence. This completes the proof.
\end{proof}
We now introduce the following new definition that extends the classical Carath\'{e}odory solution for ordinary differential equations to fractional-order dynamics with a sliding memory window, accommodating periodic conditions. It is designed for systems with a right-hand side that may be discontinuous in time but Lebesgue measurable, ensuring the existence of absolutely continuous solutions that satisfy the differential equation a.e.; cf. \cite{coddington1955theory}.

\textbf{Definition (Periodic Carath\'{e}odory Solution)}: A function $s: [0,\infty) \rightarrow [0, \infty)$ is a $T$-periodic Carath\'{e}odory solution to the FDE \eqref{eq:1DFDS} if:
\begin{enumerate}[label=(\alph*)]
\item $s$ is absolutely continuous on every compact set in $[0,\infty)$.
\item $s$ satisfies the PBC \eqref{eq:periodic1}.
\item $s$ satisfies the integral equation \eqref{eq:Volterra1} for almost all $t \in [0,T]$.
\end{enumerate}
Eq. \eqref{eq:Volterra1} defines the solution in terms of a Volterra-type integral equation with a sliding memory window and a delay term, which is the integral equivalent of the FDE \eqref{eq:1DFDS} under the CFDS. 

For existence of Carath\'{e}odory solutions, the right-hand side function $f$ typically needs to satisfy the following Carath\'{e}odory conditions, which we have already verified by Lemma \ref{lem:f_properties} and Corollary \ref{cor:f_properties}.\\[0.25em] 
\textbf{Carath\'eodory Conditions \cite{agarwal1993uniqueness}.} The function $f: [0, T] \times [0, s_{\text{in}}] \to \mathbb{R}$ is said to satisfy the Carath\'{e}odory conditions if:
\begin{enumerate}
    \item $f$ is Lebesgue measurable on $[0, T]$, for each fixed $s \in [0, s_{\text{in}}]$.
    \item For almost every fixed $t \in [0, T]$, the function $s \mapsto f(t, s)$ is continuous on $[0, s_{\text{in}}]$.
    \item There exists a Lebesgue integrable function $\mu \in L^1([0, T])$ such that for almost every $t \in [0, T]$ and for all $s \in [0, s_{\text{in}}], \left|f(t,s)\right| \le \mu(t)$.
\end{enumerate}
We define the operator $\Phi_s: X \rightarrow AC_T$ by:
\[\Phi_s(t) = s(t-L + kT) + \frac{1}{\Gamma(\alpha)} \int_{t-L}^t (t-\tau)^{\alpha-1} f(\tau, s(\tau))\,d\tau,\]
where $k$ is the smallest integer such that $t-L+k T \in [0, T]\,\forall t \ge 0$. For this operator to be well-defined on $X$, we need to ensure that $\Phi_s(X) \subseteq X$ and $\Phi_s \in AC_T$. First, we note that $f(t,0) = \vartheta^{1-\alpha} D(t) s_{in} \geq \vartheta^{1-\alpha} D_{min}s_{in} > 0$. This implies that if $s(t) = 0$ at any point, then $^{MC}_L D_t^\alpha s(t) > 0$, which means $s$ would be increasing. This prevents $s$ from becoming negative. Also, $f(t,s_{in}) = \vartheta^{1-\alpha} [D(t) - \nu(s_{in})](s_{in} - s_{in}) = 0$. This means that if $s(t) = s_{in}$ at any point, then $^{MC}_L D_t^\alpha s(t) = 0$, which prevents $s$ from exceeding $s_{in}$. For absolute continuity, note that $s \in X \subseteq AC_T$, so $s(t-L+kT)$ is absolutely continuous. The integral term:
\[\frac{1}{\Gamma(\alpha)} \int_{t-L}^t (t-\tau)^{\alpha-1} f(\tau, s(\tau))\,d\tau,\]
is absolutely continuous with respect to $t$ because $f$ is Lebesgue integrable by the Carath\'{e}odory conditions. The kernel $(t-\tau)^{\alpha-1}$ is integrable for $\alpha \in (0,1)$. Thus, the fractional integral is absolutely continuous with respect to $t$ \cite{rudin1987real}. The sum of absolutely continuous functions is absolutely continuous, so $\Phi_s$ is absolutely continuous. Furthermore, the periodic conditions \eqref{eq:periodic1} and \eqref{eq:periodic3} imply that $f$ is also $T$-periodic, and the integral over $[t+T-L, t+T]$ equals that over $[t-L, t]$. Therefore, $\Phi_s(t+T) = \Phi_s(t)$, and $\Phi_s \in X$.

The following lemma establishes the continuity of $\Phi_s$ on $X$.
\begin{lemma}\label{lem:2}
$\Phi_s$ is continuous on $X$.
\end{lemma}
\begin{proof}
Let $\{s_n\}$ be a sequence in $X$ that converges to $s \in X$ in the $AC_T$ norm. We need to show that $\Phi_{s_n} \rightarrow \Phi_s$ in the $AC_T$ norm.
For any $t \in [0,T]$:
\[
|\Phi_{s_n}(t) - \Phi_s(t)| = \left|s_n(t-L+kT) - s(t-L+kT) + \frac{1}{\Gamma(\alpha)} \int_{t-L}^t (t-\tau)^{\alpha-1} [f(\tau, s_n(\tau)) - f(\tau, s(\tau))] d\tau\right|.
\]
Since $s_n \rightarrow s$ uniformly, then for any $\epsilon > 0$, there exists $N$ such that for all $n > N: |s_n(t) - s(t)| < \epsilon\,\forall$ $t \in [0,T]$. The contribution of the delay term is therefore:
\begin{equation}
|s_n(t-L+kT) - s(t-L+kT)| \le \|s_n - s\|_\infty < \epsilon,
\end{equation}
which vanishes in the limit. Using the Lipschitz condition:
\[|f(t, s_n(t)) - f(t, s(t))| \le L_f\,|s_n(t) - s(t)|,\]
we get:
\[|\Phi_{s_n}(t) - \Phi_s(t)| \le \epsilon + \frac{L_f}{\Gamma(\alpha)} \int_{t-L}^t (t-\tau)^{\alpha-1} |s_n(\tau) - s(\tau)|\,d\tau.\]
This implies:
\begin{gather*}
|\Phi_{s_n}(t) - \Phi_s(t)| < \epsilon \left(1+\frac{L_f\,L^\alpha}{\Gamma(\alpha+1)}\right).
\end{gather*}
As $\epsilon$ can be made arbitrarily small, $\Phi_{s_n} \rightarrow \Phi_s$ uniformly. 
To complete the proof, we need to show that $\|\Phi_{s_n}' - \Phi_s'\|_{L^1} \to 0$ as $n \to \infty$. For any $t \in [0, T]$, the derivative of $\Phi_s$ at $t$ is given by:
\[
\Phi_s'(t) = s'(t-L+kT) + \frac{1}{\Gamma(\alpha)}\frac{d}{dt}\int_{t-L}^{t}(t-\tau)^{\alpha-1}f(\tau,s(\tau))d\tau.
\]
The Leibniz rule for differentiation under the integral sign gives:
\begin{equation}\label{eq:PhisP1}
\Phi_s'(t) = s'(t-L+kT) + \frac{(\alpha-1)}{\Gamma(\alpha)}\int_{t-L}^{t}(t-\tau)^{\alpha-2}f(\tau,s(\tau))d\tau - \frac{L^{\alpha-1}}{\Gamma(\alpha)} f(t-L,s(t-L)).
\end{equation}
Notice here that the upper limit term in Leibniz rule at $\tau = t, (t - \tau)^{\alpha - 1} f(\tau, s(\tau))/\Gamma(\alpha)$, is zero for $\alpha \in (0,1)$, as the singularity of the kernel $(t - \tau)^{\alpha - 1}$ at $\tau = t$ does not contribute as boundary terms when interpreting the integral in the Lebesgue or distributional sense, a standard property in fractional calculus \cite{diethelm2010analysis}. Now, for the sequence $\{s_n\}$ converging to $s$ in the $AC_T$ norm:
\begin{align*}
|\Phi_{s_n}'(t) - \Phi_s'(t)| &= \bigg|s_n'(t-L+kT) - s'(t-L+kT) + \frac{(\alpha-1)}{\Gamma(\alpha)}\int_{t-L}^{t}(t-\tau)^{\alpha-2}[f(\tau,s_n(\tau)) - f(\tau,s(\tau))]d\tau \\
&- \frac{L^{\alpha-1}}{\Gamma(\alpha)} [f(t-L,s_n(t-L)) - f(t-L,s(t-L))]\bigg|\\
&\leq |s_n'(t-L+kT) - s'(t-L+kT)| + \left|\frac{(\alpha-1)}{\Gamma(\alpha)}\int_{t-L}^{t}(t-\tau)^{\alpha-2}[f(\tau,s_n(\tau)) - f(\tau,s(\tau))]d\tau\right| \\
&+ \frac{L^{\alpha-1} L_f}{\Gamma(\alpha)} |s_n(t-L) - s(t-L)|.
\end{align*}
Using the change of variables
\[\tau = t - L y^{\frac{1}{\alpha-1}},\]
we can reduce the integral in the second term into
\[\frac{L^{\alpha-1}}{\alpha-1} \int_0^1 \left[f\left(t - L y^{\frac{1}{\alpha-1}},s_n\left(t - L y^{\frac{1}{\alpha-1}}\right)\right) - f\left(t - L y^{\frac{1}{\alpha-1}},s\left(t - L y^{\frac{1}{\alpha-1}}\right)\right)\right]\,dy.\]
Therefore,
\begin{align}
|\Phi_{s_n}'(t) - \Phi_s'(t)| &\leq |s_n'(t-L+kT) - s'(t-L+kT)| + \frac{L^{\alpha-1} L_f}{\Gamma(\alpha)} \int_0^1 \left|s_n\left(t - L y^{\frac{1}{\alpha-1}}\right) - s\left(t - L y^{\frac{1}{\alpha-1}}\right)\right|\,dy \\
&+ \frac{L^{\alpha-1} L_f}{\Gamma(\alpha)} |s_n(t-L) - s(t-L)|.\label{eq:align11}
\end{align}
Since $s_n \to s$ in the ACT norm, we have $\|s_n - s\|_{\infty} < \epsilon$ and $\|s_n' - s'\|_{L^1} < \epsilon$ for $n > N$. Therefore:
\begin{gather*}
\int_0^T |\Phi_{s_n}'(t) - \Phi_s'(t)|dt \leq \int_0^T |s_n'(t-L+kT) - s'(t-L+kT)|dt \\
+ \frac{L^{\alpha-1} L_f}{\Gamma(\alpha)} \int_0^T \int_0^1 \left|s_n\left(t - L y^{\frac{1}{\alpha-1}}\right) - s\left(t - L y^{\frac{1}{\alpha-1}}\right)\right|\,dy\,dt + \frac{L^{\alpha-1} L_f \epsilon T }{\Gamma(\alpha)}.
\end{gather*}
The first term of the right-hand-side equals $\|s_n' - s'\|_{L^1}$, which is less than $\epsilon$ by the convergence in ACT norm-- notice that the shift by $k T$ does not affect the integral over one period, so the bound holds directly. Now, consider the second term in the expression:
\[
\frac{L^{\alpha-1} L_f}{\Gamma(\alpha)} \int_0^T \int_0^1 \left|s_n\left(t - L y^{\frac{1}{\alpha-1}}\right) - s\left(t - L y^{\frac{1}{\alpha-1}}\right)\right|\,dy\,dt.
\]
Since $ |s_n(\tau) - s(\tau)| < \epsilon $ for all $ \tau \in [0, T] $ due to $ \|s_n - s\|_{\infty} < \epsilon $, and the functions are $ T $-periodic, we have:
\[
\left|s_n\left(t - L y^{\frac{1}{\alpha-1}}\right) - s\left(t - L y^{\frac{1}{\alpha-1}}\right)\right| < \epsilon \quad \text{for all } y \in [0, 1], \, t \in [0, T].
\]
Thus:
\[
\int_0^1 \left|s_n\left(t - L y^{\frac{1}{\alpha-1}}\right) - s\left(t - L y^{\frac{1}{\alpha-1}}\right)\right| \, dy < \int_0^1 \epsilon \, dy = \epsilon.
\]
Integrating over $ t $:
\[
\int_0^T \int_0^1 \left|s_n\left(t - L y^{\frac{1}{\alpha-1}}\right) - s\left(t - L y^{\frac{1}{\alpha-1}}\right)\right| \, dy \, dt < \int_0^T \epsilon \, dt = \epsilon T.
\]
So, the second term is bounded by:
\[
\frac{L^{\alpha-1} L_f}{\Gamma(\alpha)} \int_0^T \int_0^1 \left|s_n\left(t - L y^{\frac{1}{\alpha-1}}\right) - s\left(t - L y^{\frac{1}{\alpha-1}}\right)\right| \, dy \, dt < \frac{L^{\alpha-1} L_f \epsilon T}{\Gamma(\alpha)}.
\]
Therefore,
\[\|\Phi_{s_n}' - \Phi_s'\|_{L^1} < \epsilon + \frac{L^{\alpha-1} L_f \epsilon T}{\Gamma(\alpha)} + \frac{L^{\alpha-1} L_f \epsilon T}{\Gamma(\alpha)} 
= \epsilon \left( {\frac{{2T{L_f}{L^{\alpha  - 1}}}}{{\Gamma \left( \alpha  \right)}} + 1} \right).\]

Since $\epsilon$ can be made arbitrarily small, we conclude that $\|\Phi_{s_n}' - \Phi_s'\|_{L^1} \to 0$ as $n \to \infty$, which completes the proof that $\Phi_s$ is continuous on $X$ in the ACT norm.
\end{proof} 

The following lemma further establishes the compactness of $\Phi_s(X)$ in $X$.
\begin{lemma}\label{lem:3}
$\Phi_s(X)$ is relatively compact in $X$.
\end{lemma}
\begin{proof}
To show that $\Phi_s(X)$ is relatively compact in $X$, we use the Arzel\'{a}-Ascoli theorem adapted for absolutely continuous functions; cf. \cite{ambrosio2008gradient}. In particular, we need to show that $\Phi_s(X)$ is equicontinuous, equibounded, and that the derivatives of functions in $\Phi_s(X)$ are uniformly integrable. The equiboundedness follows from the fact that $\Phi_s$ maps $X$ into itself, and $X$ is bounded. For equicontinuity, we note that for any $s \in X$ and $t_1, t_2 \in [0,T]$:
\[|\Phi_s(t_1) - \Phi_s(t_2)| \le \int_{t_1}^{t_2} |\Phi_s'(t)|\,dt.\]
Since $s \in X$, we have $\|s\|_{\infty} \leq s_{\text{in}}$ and $\|s'\|_{L^1} \leq M'$ for some constant $M' > 0$, as $X$ is a bounded subset of $AC_T$. The function $f$ is bounded on $[0,T] \times [0,s_{\text{in}}]$ by Corollary \ref{cor:f_properties}, so there exists $M_f > 0$ such that $|f(t,s(t))| \leq M_f$ for all $t \in [0,T]$, $s \in X$. Specifically, since $D(t) \in [D_{\min}, D_{\max}]$, $\nu(s) \leq \mu_{\max}$, and $s_{\text{in}} - s(t) \leq s_{\text{in}}$, we can bound $f$ by
\[
|f(t,s(t))| \leq \vartheta^{1-\alpha} (D_{\max} + \mu_{\max})s_{\text{in}}.
\]
Thus, set $M_f = \vartheta^{1-\alpha} (D_{\max} + \mu_{\max})s_{\text{in}}$. To bound $\Phi_s'(t)$, we bound each term of Eq. \eqref{eq:PhisP1}:
\begin{equation}\label{eq:PhisP2}
|\Phi_s'(t)| \le |s'(t-L+kT)| + \left|\frac{(\alpha-1)}{\Gamma(\alpha)}\int_{t-L}^{t}(t-\tau)^{\alpha-2}f(\tau,s(\tau))d\tau\right| + \frac{M_f L^{\alpha-1}}{\Gamma(\alpha)}.
\end{equation}
Since $|f(t,s(t))| \leq M_f$, then the size of the integral is bounded by:
\[
\int_{t-L}^{t}(t-\tau)^{\alpha-2}|f(\tau,s(\tau))|d\tau \leq M_f \int_{t-L}^{t}(t-\tau)^{\alpha-2}d\tau = \frac{M_f\,L^{\alpha-1}}{\alpha-1}.
\]
Since $\alpha \in (0,1)$, $\alpha-1 < 0$, and the integral value is finite. Thus:
\[
|\Phi_s'(t)| \leq |s'(t-L+kT)| + \frac{2 M_f L^{\alpha-1}}{\Gamma(\alpha)}.
\]
For equicontinuity, for any $s \in X$ and $t_1, t_2 \in [0,T]$ with $t_1 < t_2$:
\[
|\Phi_s(t_1) - \Phi_s(t_2)| \leq \int_{t_1}^{t_2} |\Phi_s'(t)|dt \leq \int_{t_1}^{t_2} \left( |s'(t-L+kT)| + \frac{2 M_f L^{\alpha-1}}{\Gamma(\alpha)} \right)dt.
\]
This gives:
\[
|\Phi_s(t_1) - \Phi_s(t_2)| \leq \int_{t_1}^{t_2} |s'(t-L+kT)|dt + \frac{2 M_f L^{\alpha-1}}{\Gamma(\alpha)} (t_2 - t_1).
\]
Since $s \in X$, assume $|s'(t)| \le M''$ a.e. for some $M'' > 0$; this is possible since $X$ is bounded in $AC_T$, and absolutely continuous functions have bounded derivatives in $L^1$. Therefore,
\[\int_{t_1}^{t_2} |s'(t-L+kT)|dt \leq M'' (t_2-t_1).\]
Then:
\[
|\Phi_s(t_1) - \Phi_s(t_2)| \leq \left(M'' + \frac{2 M_f L^{\alpha-1}}{\Gamma(\alpha)}\right) (t_2 - t_1).
\]
This bound goes to 0 as $t_2 - t_1 \to 0$, ensuring equicontinuity uniformly for all $s \in X$. For uniform integrability of derivatives, we need to show that $\int_0^T |\Phi_s'(t)|dt$ is uniformly bounded for all $s \in X$. This is straightforward because:
\[
\int_0^T |\Phi_s'(t)|dt \leq \int_0^T \left( |s'(t-L+kT)| + \frac{2 M_f L^{\alpha-1}}{\Gamma(\alpha)} \right)dt \leq M' + \frac{2 M_f T L^{\alpha-1}}{\Gamma(\alpha)}.
\]
This establishes a uniform bound on $\int_0^T |\Phi_s'(t)|dt$ for all $s \in X$, proving uniform integrability of derivatives. Therefore, $\Phi_s(X)$ is relatively compact in $X$.
\end{proof}

\subsubsection{Positivity and Boundedness Properties}
\label{subsec:PosBound}
This section provides rigorous mathematical proofs for the positivity and boundedness of solutions to the FOCS. 

\begin{theorem}[Positivity of Solutions]\label{thm:POS1}
For the FOCS described in Section \ref{sec:MD1}, with initial conditions $s(0) \in (0, s_{\text{in}})$, $x(0) > 0$, and $D(t) < \mu_{\max}$ for all $t \geq 0$, the solutions $s(t)$ and $x(t)$ remain positive for all $t > 0$ such that $s(t) \in (0, s_{\text{in}})$ for all $t > 0$.
\end{theorem}
\begin{proof}
We have already established in Section \ref{subsec:WPFOCS} that if $s(t) = 0$ at any point, then $s$ would be increasing. Also, if $s(t) = s_{\text{in}}$ at any point, then $s$ would be non-increasing. These boundary behaviors, combined with the continuity of the solution established by the Carath\'{e}odory framework, ensure that $s(t) \in [0, s_{\text{in}}]$ for all $t > 0$. Let $s(t)$ be any $T$-periodic Carath\'{e}odory solution to the FOCS. By construction, $s(t) \in X$, which implies that $0 \leq s(t) \leq s_{\text{in}}$ for all $t \in [0,T]$, and, by periodicity, for all $t > 0$.
This implies:
\[x(t) = Y(s_{\text{in}} - s(t)) \geq Y(s_{\text{in}} - s_{\text{in}}) = 0.\]
To show that $x(t) > 0$ strictly, notice first that $x(0) > 0$ when $s(0) < s_{\text{in}}$ from the relation $x(0) = Y(s_{\text{in}} - s(0))$. Suppose now, for contradiction, that there exists a time $t_1 > 0$ such that $s(t_1) = s_{\text{in}}$. By continuity of $s(t)$, there must exist a time interval $(t_0, t_1)$ such that $s(t) < s_{\text{in}}$ for all $t \in (t_0, t_1)$ and $\lim_{t \to t_1^-} s(t) = s_{\text{in}}$. From the reduced one-dimensional FDE \eqref{eq:1DFDS}, as $t \to t_1, s(t) \to s_{\text{in}}$, which means $s_{\text{in}} - s(t) \to 0$. For $s(t_1) \to s_{\text{in}}, {}^{\text{MC}}_{L}D_t^\alpha s(t)$ must be positive in a neighborhood of $t_1$. However, as $s(t) \to s_{\text{in}}$, we have:
\[\nu(s(t)) = \frac{\mu_{\max} s(t)}{K Y (s_{\text{in}} - s(t)) + s(t)} \to \mu_{\max}.\]
For the FD to remain positive as $s(t) \to s_{\text{in}}$, we would need $D(t) > \mu_{\max}$ in a neighborhood of $t_1$, as indicated by the FDE \eqref{eq:1DFDS}. However, since $D(t) < \mu_{\max}$ for all $t \geq 0$ by assumption, we have $D(t) - \nu(s(t)) < 0$ as $s(t) \to s_{\text{in}}$, implying ${}^{\text{MC}}_{L}D_t^\alpha s(t) < 0$. This contradicts the requirement that ${}^{\text{MC}}_{L}D_t^\alpha s(t) > 0$ for $s(t)$ to approach $s_{\text{in}}$. Thus, $s(t_1) = s_{\text{in}}$ cannot occur. Therefore, $s(t) < s_{\text{in}}$ for all $t > 0$ if $s(0) < s_{\text{in}}$ (equivalently, if $x(0) > 0$). From the relation $x(t) = Y(s_{\text{in}} - s(t))$, we conclude that $x(t) > 0$ for all $t > 0$ if $x(0) > 0$. This completes the proof of strict positivity for the solutions of the FOCS.
\end{proof}

Building on the positivity results, we now establish the boundedness of solutions.

\begin{corollary}[Boundedness of Solutions]\label{cor:BOS1} 
For the FOCS described in Section \ref{sec:MD1}, with initial conditions $s(0) \in [0, s_{\text{in}}] $ and $ x(0) \ge 0 $, the solutions $ s(t) $ and $ x(t) $ remain bounded for all $ t > 0 $. Specifically, $ 0 \leq s(t) \leq s_{\text{in}} $ and $ 0 \le x(t) \leq Y s_{\text{in}} $ for all $ t > 0 $, and the solution trajectory lies within the invariant set 
\begin{equation}\label{eq:InVS1}
\Omega = \{(s, x) \in \mathbb{R}^2 : 0 \leq s \leq s_{\text{in}}, 0 \le x \leq Y s_{\text{in}}\}, 
\end{equation}
forming a closed orbit.
\end{corollary}
\begin{proof}
Let $s \in X$ be any $T$-periodic Carath\'{e}odory solution to the FOCS. Since $s$ is $T$-periodic, the bound $0 \leq s(t) \leq s_{\text{in}}$ extends to all $t > 0$. It follows that:
\[
0 \leq x(t) = Y (s_{\text{in}} - s(t)) \leq Y s_{\text{in}} \quad \forall t > 0.
\]
Specifically, when $s(0) < s_{\text{in}}$, then $x(0) = Y (s_{\text{in}} - s(0)) > 0 $. In this case, Theorem \ref{thm:POS1} establishes that $s(t) < s_{\text{in}}$ for all $t > 0$, and $x(t) > 0$ for all $t > 0$. Thus, the solution trajectory $(s(t), x(t))$ remains within the set $\Omega = \{(s, x) \in \mathbb{R}^2 : 0 \leq s \leq s_{\text{in}}, 0 \le x \leq Y s_{\text{in}}\}$. The periodicity of $s$ and $x$ implies that the trajectory forms a closed orbit within $\Omega$. Hence, $s$ and $x$ are bounded for all $ t > 0 $.
\end{proof}
Corollary \ref{cor:BOS1} ensures that the system's trajectories remain confined to a biologically feasible region, satisfying the physical constraints of the chemostat model. 

\subsubsection{Existence of Non-trivial Periodic Carath\'{e}odory Solutions}
\label{subsec:EPCS1}
We are ready now to prove the existence of non-trivial, periodic solutions for the FOCS under periodic dilution rates strategies using fixed-point theorems. 

\begin{theorem}[Existence of Periodic Carath\'{e}odory Solutions]\label{thm:existence}
Suppose that the assumptions of Theorem \ref{thm:POS1} hold true. Then the FOCS governed by the one-dimensional FDE \eqref{eq:1DFDS} with the periodic conditions \eqref{eq:periodic1} and \eqref{eq:periodic3} admits at least one non-trivial, $T$-periodic Carath\'{e}odory solution.
\end{theorem}
\begin{proof}
By Schauder's fixed-point theorem, Lemmas \ref{lem:2} and \ref{lem:3} imply that $\Phi_s$ has at least one fixed point in $X$. By Theorem \ref{thm:POS1}, this fixed point is a non-trivial, $T$-periodic Carath\'{e}odory solution to our FDE with piecewise continuous dilution rate. 
\end{proof}

\subsubsection{Uniqueness of Periodic Carath\'{e}odory Solutions}
\label{subsec:UPCS1}
We begin this section by establishing a key monotonicity property of the nonlinear term $f$, which is essential for proving uniqueness of the periodic solution under certain conditions.

\begin{lemma}[Strict Monotonicity of the Nonlinear Term]\label{lem:monotonicity}
Let $K Y > 1$ and assume that $D(t) > \mu_{\max}\,\forall t \ge 0$. Then the function $f$ is strictly decreasing in $s$ on the interval $[0, s_{\text{in}}]$ for each fixed $t$.
\end{lemma}
\begin{proof}
First, we analyze the second derivative of $\nu(s)$:
\[
\nu''(s) = \frac{2KY\mu_{\max}s_{\text{in}}(KY - 1)}{[KY(s_{\text{in}} - s) + s]^3}.
\]
Since $KY > 1$, $\mu_{\max} > 0$, $s_{\text{in}} > 0$, and $s \in [0, s_{\text{in}}]$, the denominator is positive. Thus, $\nu''(s) > 0$, confirming that $\nu(s)$ is strictly convex on $[0, s_{\text{in}}]$. Next, consider the derivative of $f$ with respect to $s$:
\[
\frac{\partial f}{\partial s} = \vartheta^{1-\alpha} \left[\nu(s) - D(t) - (s_{\text{in}} - s)\nu'(s)\right].
\]
To show that $\displaystyle{\frac{\partial f}{\partial s}} < 0$, we need to prove that $\nu(s) - (s_{\text{in}} - s)\nu'(s) < D(t)$. Let $g(s) = \nu(s) - (s_{\text{in}} - s)\nu'(s)$. We compute its derivative:
\[
g'(s) = 2\nu'(s) - (s_{\text{in}}-s)\nu''(s).
\]
Plugging the expressions for $\nu'(s)$ and $\nu''(s)$ into $g'(s)$ and simplifying:
\[
g'(s) = \frac{2KY\mu_{\max}s_{\text{in}}^2}{[KY(s_{\text{in}} - s) + s]^3}.
\]
Since $KY > 1$, $\mu_{\max} > 0$, and $s_{\text{in}} > 0$, it follows that $g'(s) > 0$ for all $s \in [0, s_{\text{in}}]$. This demonstrates that $g(s)$ is strictly increasing on $[0, s_{\text{in}}]$. Hence, its supremum on this interval is $\mu_{\max}$. For $\displaystyle{\frac{\partial f}{\partial s}} < 0$, we require $g(s) < D(t)$ for all $s \in [0, s_{\text{in}}]$, which is satisfied if $D(t) > \mu_{\max}$.
\end{proof}

Let $\mathbb{Z}^+$ denote the set of positive integers. We are ready now to prove the uniqueness of periodic solutions for the FOCS using fixed-point theorems, ensuring its well-posedness under periodic dilution rates strategies. 

\begin{theorem}[Uniqueness of Periodic Carath\'{e}odory Solutions]\label{thm:uniqueness}
Consider the FOCS governed by the one-dimensional FDE \eqref{eq:1DFDS} with the periodic conditions \eqref{eq:periodic1} and \eqref{eq:periodic3}. The following statements hold true:
\begin{enumerate}[label=(\roman*)]
\item The trivial solution is the unique $T$-periodic Carath\'{e}odory solution if $L \ge T, K Y > 1$ and $D(t) > \mu_{\max}$ for all $t \ge 0$.
\item There exists a unique non-trivial, $T$-periodic Carath\'{e}odory solution if any of the following two conditions hold: 
\begin{enumerate}[label=(\alph*)]
\item $s(0) \leq s^* = \displaystyle{\frac{s_{\text{in}} \sqrt{K Y}}{\sqrt{K Y} + 1}}$, and $D(t) \leq \nu(s^*)$ for all $t \in [0, T]$.
\item $s(0) \leq s^*$, and the average dilution rate satisfies $\bar{D} = \frac{1}{T} \int_0^T D(t) \, dt \leq \nu(s^*)$.
\end{enumerate}
\end{enumerate}
\end{theorem}
\begin{proof}
To prove uniqueness of solutions, assume there exist two distinct $T$-periodic Carath\'{e}odory solutions $s_1, s_2 \in X$. Define the difference $\delta(t) = s_1(t) - s_2(t): \| \delta \|_{\infty} = \sup_{t \in [0,T]} |\delta(t)| = M > 0$. Since $\delta(t)$ is absolutely continuous and $T$-periodic, there exists $t^* \in [0, T]: |\delta(t^*)| = M$ and 
\begin{equation}\label{eq:delta}
\delta(t^*) = \delta(t^* - L + kT) + \frac{1}{\Gamma(\alpha)} \int_{t^*-L}^{t^*} (t^* - \tau)^{\alpha-1} [f(\tau, s_1(\tau)) - f(\tau, s_2(\tau))] \, d\tau,
\end{equation}
where $k$ is the smallest integer such that $t^* - L + kT \in [0, T]$. 
\begin{enumerate}[label=(\roman*)]
\item We consider the following two cases:
\begin{description}
\item[Case 1] ($L > T: L \neq n T$, for some $n \in \mathbb{Z}^+$). Consider the following two subcases:
\begin{description}
\item[Subcase 1.1] ($\delta(t^*) = M > 0$). Here, $s_1(t^*) > s_2(t^*)$. Since $\delta(t^*)$ is at its maximum value $M$, the FD must satisfy $^{MC}_L D_t^\alpha \delta(t^*) \leq 0$. But from the original FDE \eqref{eq:1DFDS}, we have:
\[^{MC}_L D_t^\alpha \delta(t^*) = f(t^*, s_1(t^*)) - f(t^*, s_2(t^*)).\]
This shows that $f(t^*, s_1(t^*)) \le f(t^*, s_2(t^*))$, which implies that the integral in \eqref{eq:delta} must be non-positive. Moreover, by Lemma \ref{lem:monotonicity}, $f(t, s)$ is strictly decreasing in $s$, so $s_1(\tau) > s_2(\tau)$ implies $f(\tau, s_1(\tau)) < f(\tau, s_2(\tau))$. Define $A = \{ \tau \in [t^* - L, t^*] \mid s_1(\tau) > s_2(\tau) \}$ and $B = \{ \tau \in [t^* - L, t^*] \mid s_1(\tau) \leq s_2(\tau) \}$. The integral in \eqref{eq:delta} splits as:
\begin{align*}
\int_{t^*-L}^{t^*} (t^* - \tau)^{\alpha-1} [f(\tau, s_1(\tau)) - f(\tau, s_2(\tau))] \, d\tau &= \int_A (t^* - \tau)^{\alpha-1} [f(\tau, s_1(\tau)) - f(\tau, s_2(\tau))] \, d\tau\\
&+ \int_B (t^* - \tau)^{\alpha-1} [f(\tau, s_1(\tau)) - f(\tau, s_2(\tau))] \, d\tau.
\end{align*}
Since $(t^* - \tau)^{\alpha-1} > 0$ for $\tau \in [t^* - L, t^*)$ and $\alpha \in (0,1)$, the integrand is negative on $A$ and non-negative on $B$. If the integral is negative, then:
\[
M = \delta(t^*) < \delta(t^* - L + kT) \leq M,
\]
a contradiction. If the integral is zero, then $f(\tau, s_1(\tau)) = f(\tau, s_2(\tau))$ a.e., implying $s_1(\tau) = s_2(\tau)$ a.e. in $[t^* - L, t^*]$. Since $L > T$ (by assumption), the sliding memory window $[t-L,t]$ contains at least one complete period of any $T$-periodic solution, then by $T$-periodicity, $s_1 = s_2$ globally, contradicting distinctness.
\item[Subcase 1.2] ($\delta(t^*) = -M < 0$). We derive the proof by contradiction via a similar argument to that of \textbf{Subcase 1.1}. Notice here that $s_2(t^*) > s_1(t^*)$. Since $\delta(t^*)$ is at its minimum value $-M$, the FD must satisfy $^{MC}_L D_t^\alpha \delta(t^*) \geq 0$. But from the original FDE \eqref{eq:1DFDS}, we have:
\[^{MC}_L D_t^\alpha \delta(t^*) = f(t^*, s_1(t^*)) - f(t^*, s_2(t^*)).\]
This shows that $f(t^*, s_1(t^*)) \geq f(t^*, s_2(t^*))$, which implies that the integral in \eqref{eq:delta} must be non-negative. Moreover, by Lemma \ref{lem:monotonicity}, $f(t, s)$ is strictly decreasing in $s$, so $s_2(\tau) > s_1(\tau)$ implies $f(\tau, s_1(\tau)) > f(\tau, s_2(\tau))$. Define $A' = \{ \tau \in [t^* - L, t^*] \mid s_2(\tau) > s_1(\tau) \}$ and $B' = \{ \tau \in [t^* - L, t^*] \mid s_1(\tau) \geq s_2(\tau) \}$. The integral in \eqref{eq:delta} splits as:
\begin{align*}
\int_{t^*-L}^{t^*} (t^* - \tau)^{\alpha-1} [f(\tau, s_1(\tau)) - f(\tau, s_2(\tau))] \, d\tau &= \int_{A'} (t^* - \tau)^{\alpha-1} [f(\tau, s_1(\tau)) - f(\tau, s_2(\tau))] \, d\tau\\
&+ \int_{B'} (t^* - \tau)^{\alpha-1} [f(\tau, s_1(\tau)) - f(\tau, s_2(\tau))] \, d\tau.
\end{align*}
Since $(t^* - \tau)^{\alpha-1} > 0$ for $\tau \in [t^* - L, t^*)$ and $\alpha \in (0,1)$, the integrand is positive on $A'$ and non-positive on $B'$. If the integral is positive, then:
\[
-M = \delta(t^*) > \delta(t^* - L + kT) \geq -M,
\]
a contradiction. If the integral is zero, then $f(\tau, s_1(\tau)) = f(\tau, s_2(\tau))$ a.e., implying $s_1(\tau) = s_2(\tau)$ a.e. in $[t^* - L, t^*]$. Since the sliding memory window $[t-L,t]$ contains at least one complete period of any $T$-periodic solution, then by $T$-periodicity, $s_1 = s_2$ globally, contradicting distinctness.
\end{description}
Hence, no nonzero $\delta(t)$ can exist, proving the uniqueness of the $T$-periodic Carath\'{e}odory solution.
\item[Case 2] ($L = n T$, for some $n \in \mathbb{Z}^+$). The condition $L = n T$ ensures that the sliding memory window aligns perfectly with the periodic structure. Since $\delta(t)$ is $T$-periodic, we have $\delta(t - L + k T) = \delta(t)$. This simplifies the integral form of the FDE:
\begin{align*}
\delta(t) &= \delta(t - L + k T) + \frac{1}{\Gamma(\alpha)} \int_{t-L}^{t} (t - \tau)^{\alpha-1} [f(\tau, s_1(\tau)) - f(\tau, s_2(\tau))] \, d\tau\\
\Rightarrow 0 &= \frac{1}{\Gamma(\alpha)} \int_{t-n T}^{t} (t - \tau)^{\alpha-1} [f(\tau, s_1(\tau)) - f(\tau, s_2(\tau))] \, d\tau.
\end{align*}
Since $(t - \tau)^{\alpha-1} > 0$ for $\tau \in [t - n T, t), \alpha \in (0,1)$, and $f(t, s)$ is strictly decreasing in $s$, the integral being zero implies $f(\tau, s_1(\tau)) = f(\tau, s_2(\tau))$ a.e. in $[t - n T, t]$. Therefore, $s_1(\tau) = s_2(\tau)$ a.e. in $[t - n T, t]$. Since $s_1$ and $s_2$ are continuous (as Carath\'{e}odory solutions are continuous), then $s_1(\tau) = s_2(\tau)$ everywhere.
\end{description}
\item \begin{enumerate}[label=(\alph*)]
\item Assume $s(0) \leq s^*$, and $D(t) \leq \nu(s^*) = \displaystyle{\frac{\mu_{\max}}{1 + \sqrt{K Y}}}$ for all $t \in [0, T]$. First, we show that $s(t) \leq s^*$ for all $t$. Suppose there exists $t_1 > 0$ such that $s(t_1) > s^*$. Let $t_0 = \inf \{ t > 0 \mid s(t) > s^* \}$. Since $s(t)$ is continuous (by the Carath\'{e}odory property), $s(t_0) = s^*$. At $t_0$:
\[
{}_L^{MC} D_t^\alpha s(t_0) = \vartheta^{1-\alpha} [D(t_0) - \nu(s^*)](s_{\text{in}} - s^*) \leq 0,
\]
since $D(t_0) \leq \nu(s^*)$. For an absolutely continuous function, if ${}_L^{MC} D_t^\alpha s(t_0) \leq 0$ at a point where $s(t_0) = s^*$, the solution cannot increase beyond $s^*$ at $t_0$. Thus, $s(t) \leq s^*$ for all $t$. Since $s(0) \leq s^*$, we have $s(t) \in [0, s^*]$. Now, assume that the two distinct $T$-periodic Carath\'{e}odory solutions $s_1, s_2 \in X$ with $s_1(t), s_2(t) \in [0, s^*]$. 

\textbf{Case 1 ($\delta(t^*) = M > 0$).} Here $s_1(t^*) > s_2(t^*)$, and $^{MC}_L D_t^\alpha \delta(t^*) \leq 0$, so $f(t^*, s_1(t^*)) \leq f(t^*, s_2(t^*))$. Now, consider
\[
f(t, s_1) - f(t, s_2) = \vartheta^{1-\alpha} \left[-D(t) \delta(t) - [h(s_1(t)) - h(s_2(t))]\right],
\]
where $h(s) = \nu(s) (s_{\text{in}} - s)$. Since
\[
h'(s) = \mu_{\max} \frac{K Y (s_{\text{in}} - s)^2 - s^2}{[K Y (s_{\text{in}} - s) + s]^2} > 0,
\]
for $s \in [0, s^*)$, then $h(s)$ is strictly increasing for $s \in [0, s^*)$ with $h'(s^*) = 0$. For any $s_1, s_2 \in [0, s^*]$, having $s_1(t^*) > s_2(t^*)$, implies $h(s_1(t^*)) > h(s_2(t^*))$, so:
\[
f(t^*, s_1(t^*)) - f(t^*, s_2(t^*)) = \vartheta^{1-\alpha} \left[-D(t^*) M - [h(s_1(t^*)) - h(s_2(t^*))]\right] < 0,
\]
which implies that the integral in \eqref{eq:delta} is negative. If $0 < L \neq n T$, for any $n \in \mathbb{Z}^+$, then:
\[
M = \delta(t^*) < \delta(t^* - L + kT) \le M,
\]
a contradiction. Otherwise, if $L = n T$, then the integral in \eqref{eq:delta} vanishes, another contradiction.

\textbf{Case 2 ($\delta(t^*) = -M < 0$).} Here $s_2(t^*) > s_1(t^*)$, and $^{MC}_L D_t^\alpha \delta(t^*) \geq 0$, so $f(t^*, s_1(t^*)) \geq f(t^*, s_2(t^*))$. Since $h(s_1(t^*)) < h(s_2(t^*))$, we have:
\[
f(t^*, s_1(t^*)) - f(t^*, s_2(t^*)) = \vartheta^{1-\alpha} \left[-D(t^*)(-M) - [h(s_1(t^*)) - h(s_2(t^*))]\right] > 0,
\]
which implies that the integral in \eqref{eq:delta} is positive. If $0 < L \neq n T$, for any $n \in \mathbb{Z}^+$, then:
\[
-M = \delta(t^*) > \delta(t^* - L + kT) \ge -M,
\]
a contradiction. Otherwise, if $L = n T$, then the integral in \eqref{eq:delta} collapses, another contradiction.

Hence, no nonzero $\delta(t)$ can exist, proving the uniqueness of the $T$-periodic Carath\'{e}odory solution.

\item Assume $s(0) \leq s^*$, and $\displaystyle{\bar{D} = \frac{1}{T} \int_0^T D(t) \, dt \leq \nu(s^*)}$. As in (a), $s(t) \leq s^*$ for all $t$, since $D(t) \leq \displaystyle{\frac{\mu_{\max}}{1 + \sqrt{K Y}}} < \mu_{\max}$ in some intervals and the average constraint ensures the periodic solution aligns with an equilibrium $\bar{s} \leq s^*$. Suppose that the two distinct solutions $s_1, s_2 \in X$ with $s_1(t), s_2(t) \in [0, s^*]$. 

\textbf{Case 1 ($\delta(t^*) = M > 0$).} Then $^{MC}_L D_t^\alpha \delta(t^*) \leq 0$, so $f(t^*, s_1(t^*)) \leq f(t^*, s_2(t^*))$. Since $h(s_1(t^*)) > h(s_2(t^*))$, the contradiction arises as in (a). 

\textbf{Case 2 ($\delta(t^*) = -M < 0$).} Similarly, $f(t^*, s_1(t^*)) \geq f(t^*, s_2(t^*))$, leading to a contradiction as in (a). 

The average constraint ensures $\delta(t) = 0$ globally, proving the uniqueness of the $T$-periodic Carath\'{e}odory solution.
\end{enumerate}
\end{enumerate}
\end{proof}
Theorem \ref{thm:uniqueness} shows that the FDE \eqref{eq:1DFDS} with sliding memory, piecewise continuous dilution rate, and $T$-PBCs has a unique $T$-periodic Carath\'{e}odory solution when sliding memory length $L$, dilution rate $D$, the system parameters $K$ and $Y$, or the initial substrate concentration satisfy appropriate conditions. The existence and uniqueness results derived in this work validate the proposed mathematical model in Section \ref{sec:MD1} for optimizing periodic dilution rate strategies in bioprocesses for water treatment, ensuring that the model is well-posed and has a unique solution even with piecewise continuous dilution rate inputs when the parameters satisfy appropriate conditions, which is the realistic case in practical applications.

\begin{remark}
Theorem \ref{thm:uniqueness} provides a rigorous foundation for the uniqueness of non-trivial, periodic solutions to the reduced one-dimensional FOCS. The result is significant for ensuring the well-posedness of the model, particularly under PBCs and sliding memory effects. The theorem presents three distinct sets of sufficient conditions for uniqueness, covering a wide range of practical scenarios. The sufficient conditions accommodate cases where periodic memory lengths align with the system's cycle, and address other general cases of non-aligned memory windows under certain conditions on the system parameters $K$ and $Y$, the initial substrate concentration, and the dilution rate. The use of the Carath\'{e}odory framework is particularly suitable here, as it accommodates piecewise continuous dilution rates and ensures solution regularity even in the presence of discontinuities in the dilution rate.  The theorem is generally constructive in that it identifies practical design conditions (e.g., bounds on the dilution rate or memory length) that guarantee uniqueness. This is especially important for applications in bioreactor optimization, where dilution rate strategies must ensure predictable and stable system behavior.
\end{remark}

\begin{remark}
Condition (i) of Theorem \ref{thm:uniqueness} typically leads to the washout state, as no microbial growth can be sustained when the dilution rate exceeds the maximum growth rate. That is, the FOCS has the unique solution $(s(t), x(t)) = (s_{\text{in}}, 0)$ for all $t \ge 0$.
\end{remark}

\begin{remark}
Without $ K Y > 1 $ in Condition (i) of Theorem \ref{thm:uniqueness}, $f$ may not be strictly decreasing, allowing multiple periodic solutions, as the contradiction in the integral equation fails. Furthermore, the condition $ L \geq T $ is also mathematically crucial here, as it ensures the memory window covers a full period, allowing local equality to imply global equality directly. For $L < T$, extending local equality into global equality is flawed, as non-zero integrals at points where $ |\delta(t)| < \|\delta\|_{\infty} $ do not consistently produce contradictions, potentially permitting distinct solutions. Thus, uniqueness may not hold for $L <T $ under the given conditions. While $D(t) > \mu_{\max }$ strongly suggests the washout state is the only stable solution, our mathematical proof of uniqueness here requires strict monotonicity as well as $L \ge T$.
\end{remark}

\begin{remark}
It is noteworthy to mention that at a discontinuity point $t_d$ of $D(t)$, the right-hand side function $f$ experiences a jump. The solution $s$ remains continuous across $t_d$ but may have a corner. The FD $^{MC}_L D_t^\alpha s(t)$ is still well-defined at such points in the Carath\'{e}odory sense, as it is interpreted through the integral equation. The sliding memory window $[t-L,t]$ helps to smooth out the effects of discontinuities, as the CFDS considers the weighted average of past states. 
\end{remark}

Following the rigorous mathematical analysis of the reduced FOCS, encompassing the examination of trivial and non-trivial equilibria and the establishment of well-posedness, the next section summarizes the key contributions and implications of this research. Furthermore, it discusses potential avenues for future advancements in this field.

\section{Numerical Simulations}
\label{sec:numerical_simulations}
To validate the theoretical results presented in this paper, we conduct a series of numerical simulations designed to test the existence, uniqueness, positivity, boundedness, and stability of periodic solutions for the FOCS with sliding memory and PBCs. These simulations focus on the reduced one-dimensional FDE \eqref{eq:1DFDS} and the full two-dimensional system \eqref{eq:sysdyn1} and \eqref{eq:sysdyn2}, using the FG-PS method of \citet{elgindy2024fourierA,elgindy2024fourierB} for numerical discretization. The resulting nonlinear system of equations is solved using MATLAB's \texttt{fsolve} function. For smooth dilution rates, we collocate the substrate concentrations at equally spaced nodes given by $t_j = j T/N$, for $j = 0, \ldots, N-1$, with $N = 100$. The collocated solutions are then interpolated at equally spaced nodes $\hat{t}_j = j\Delta t$, for $j = 0, \ldots, M$, where $\Delta t = T/M$ and $M = 200$. For discontinuous dilution rates, we increase the collocation and interpolation mesh grid sizes into $N = 300$ and $M = 400$ to obtain more accurate solutions. Below, we present several numerical simulations, each targeting specific claims from Theorems \ref{thm:existence}, \ref{thm:POS1}, \ref{thm:uniqueness}, and Section \ref{subsec:NTEFODS}. 

Now, consider the FOCM parameter and dilution rate dataset: 
\[\mathcal{D} = \{\alpha = 0.8, L = 1.5, T = \vartheta = s_{\text{in}} = Y = K = 1, \mu_{\max} = 3.1, D(t) = 1 + 0.5 \sin(2\pi t/T)\,\forall t \ge 0\}.\] 
Since $D(t) \in [0.5, 1.5]$. Since $\bar{D} = 1$, the equilibrium substrate concentration is $\bar{s} \approx 0.323$.

Figure \ref{fig:Ex11} presents the numerical results of solving the FDE, highlighting the dynamic interplay between substrate, biomass, and dilution rate over a single period. The displayed results capture the periodic and memory-dependent behavior characteristic of the FOCS and provide insights into the stability and efficiency of continuous bioreactor operations under variable conditions. In particular, Figure \ref{fig:Ex11}(a) shows that the substrate concentration $s$ exhibits clear periodic behavior with period $T = 1$, oscillating around the equilibrium substrate concentration $\bar{s} = 0.323$, consistent with the average dilution rate $\bar{D} = 1$. The smooth oscillatory profile reflects the system's response to the periodic dilution rate $D$, depicted in Figure \ref{fig:Ex11}(b), which varies within $[0.5, 1.5]$. This periodicity confirms the effectiveness of the CFDS, ${}_L^{MC}D_t^\alpha$, in preserving the periodic nature of the solution---a key advantage over standard FDs that may not maintain periodicity for non-constant solutions. The phase relationship between $D(t)$ and $s(t)$ reflects the delayed microbial response due to memory effects, a feature captured by the fractional-order model. Figure \ref{fig:Ex11}(c) illustrates the corresponding biomass concentration $x$, computed via the relation~\eqref{eq:biomass}, which also exhibits $T$-periodic behavior and remains strictly positive, aligning with the theoretical positivity results established in Theorem \ref{thm:POS1}. The absolute residual error plot in Figure \ref{fig:Ex11}(d) demonstrates the numerical accuracy of the solution, with absolute residuals at collocation points remaining of order $10^{-14}$, indicating that the FG-PS method effectively captures the system's dynamics. The confinement of $s$ within $[0, s_{\text{in}}] = [0, 1]$ further corroborates the boundedness properties outlined in Corollary \ref{cor:BOS1}, which ensures that the solution trajectory remains within the biologically feasible invariant set $\Omega$. Since $0 < s(0) \approx 0.274 < s_{\text{in}} = 1$ and $D(t) < \mu_{\max} = 3.1$, this example also numerically verifies Theorem \ref{thm:existence}, which asserts the existence of at least one non-trivial, $T$-periodic Carath\'{e}odory solution for the FDE~\eqref{eq:1DFDS} with PBCs \eqref{eq:periodic1} and~\eqref{eq:periodic3}. Furthermore, the conditions $s(0) < s^* = 0.5$ and $D(t) < \nu(s^*) = 1.55$ numerically verifies Theorem \ref{thm:uniqueness}, which asserts the uniqueness of that non-trivial, $T$-periodic Carath\'{e}odory solution.

From a biological perspective, the periodic oscillations in substrate and biomass concentrations reflect the adaptive behavior of microbial populations in response to cyclic nutrient supply regulated by the dilution rate. The Contois growth model, embedded in the FOCS, captures the inhibitory effects of elevated biomass concentrations, making it suitable for modeling microbial processes in wastewater treatment. The positivity of $x(t)$ throughout the time interval indicates sustained microbial activity, which is essential for effective pollutant removal. The memory-dependent formulation, enabled by the CFDS, accounts for the cumulative influence of historical nutrient availability, aligning with the observation that microbial growth responds to recent environmental conditions. The periodic behavior around $\bar{s} \approx 0.323$ suggests a stable operational regime conducive to efficient pollutant degradation under periodic dilution rate.

\begin{figure}[H]
\centering
\includegraphics[scale=0.65]{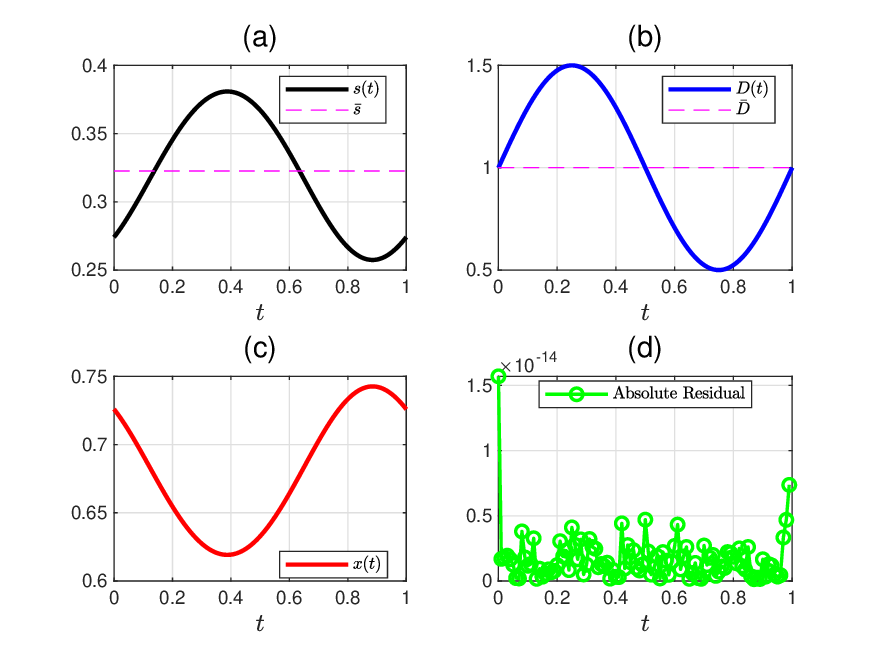}
\caption{Plots illustrating the results of numerically solving the chemostat FDE \eqref{eq:1DFDS} with Contois growth function. The panels display: (a) the approximate interpolated substrate concentration profile, (b) the dilution rate, (c) the corresponding approximate biomass concentration, and (d) the absolute residual error at collocation points, over one period $T$.
}
\label{fig:Ex11}
\end{figure}

Figure \ref{fig:Ex12} shows further the approximate substrate concentration profiles $s(t)$ over one period $T$ obtained by numerically solving the FOCM with Contois growth function for $100$ different random initial guesses, under the same configuration set $\mathcal{D}$. Notably, the resulting curves are either identical to the substrate concentration profile previously obtained in Figure \ref{fig:Ex11}, or are horizontal straight lines corresponding to the trivial solution $s(t) = s_{\text{in}}$, as proven by Theorem \ref{thm:uniqueness}. 

The numerical simulations demonstrate that, for a wide range of initial substrate concentrations, the solutions of the FOCS converge to a single, non-trivial, periodic substrate concentration profile. This convergence highlights the stability of the substrate dynamics, driven by the periodic dilution rate and the Contois growth function, which effectively captures microbial inhibition at high biomass levels. Such stability reflects the robust adaptability of the microbial system to cyclic nutrient inputs, ensuring consistent pollutant degradation in wastewater treatment applications. However, a small subset of solutions converges to the trivial solution, $s(t) = s_{\text{in}}$, corresponding to washout scenarios where maximum initial substrate concentrations result in zero biomass concentrations, preventing any pollutant degradation. These findings underscore the importance of selecting initial substrate concentrations that promote sustained microbial activity to prevent washout and optimize treatment efficiency.

\begin{figure}[H]
\centering
\includegraphics[scale=0.5]{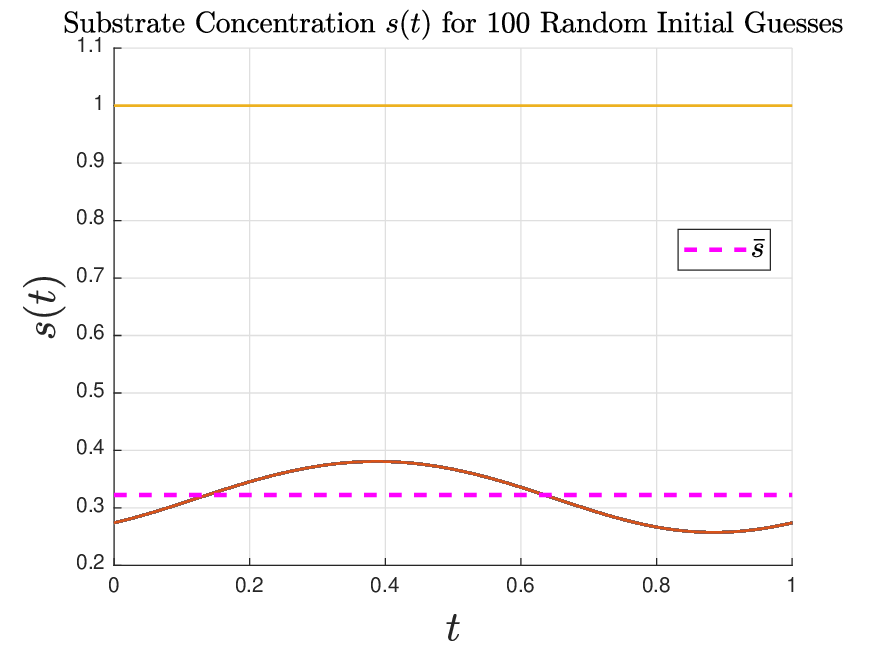}
\caption{Approximate substrate concentration profiles $s(t)$ over one period $T$ obtained by numerically solving the FOCM with Contois growth function for $100$ different random initial guesses. These initial guesses were generated by sampling each element independently and uniformly at random from the interval $[0, s_{\text{in}}]$. Each curve corresponds to a distinct initial substrate vector. The magenta dashed line represents the steady-state substrate concentration $\bar s$.}
\label{fig:Ex12}
\end{figure}

Figure~\ref{fig:Ex14} illustrates the time evolution of the substrate concentration $s(t)$ for various fractional orders $\alpha \in [0.1,\,1]$, with all simulations initialized at the steady-state value $\bar{s}$. The dashed black curve represents the classical integer-order case $\alpha = 1$, computed using MATLAB's boundary value problem solver \texttt{bvp4c}. For smaller values of $\alpha$, the trajectories exhibit higher oscillation amplitudes and slower convergence to the steady state, reflecting the strong memory effect inherent in the system. As $\alpha$ increases, the oscillations diminish, and the system stabilizes more rapidly. This behavior underscores the critical role of the fractional order $\alpha$ in shaping transient dynamics: lower values prolong memory effects and delay system adaptation, while higher values lead to faster relaxation. These insights are particularly relevant in the context of wastewater treatment, where microbial responses to nutrient fluctuations may be delayed due to adaptation mechanisms or metabolic lag, especially under variable influent conditions.

All simulated trajectories remain bounded and oscillate around the steady state, demonstrating both the numerical stability of the method and the physical plausibility of the model. As $\alpha$ approaches unity, the FDE converges to its classical integer-order counterpart, characterized by rapid stabilization of $s(t)$. However, such behavior may be less representative of real wastewater systems, where historical substrate levels influence microbial dynamics. Therefore, the FOCM provides a more realistic framework for capturing memory-dependent behaviors.

The bounded oscillations around $\bar{s} \approx 0.323$ ensure that substrate concentrations remain within biologically meaningful ranges, preventing pollutant accumulation. This intrinsic stability, as governed by the FDE, supports consistent water quality in continuous treatment processes. Additionally, the condition $x(t) > 0$ indicates sustained microbial activity throughout the simulation.

Biologically, the fractional order $\alpha$ serves as a tuning parameter that modulates memory-dependent substrate dynamics, offering a valuable tool for optimizing bioreactor performance. Lower values of $\alpha$ may be advantageous for systems treating complex or recalcitrant pollutants requiring prolonged microbial adaptation, while higher values may improve operational efficiency under stable influent conditions.

\begin{figure}[H]
\centering
\includegraphics[scale=0.5]{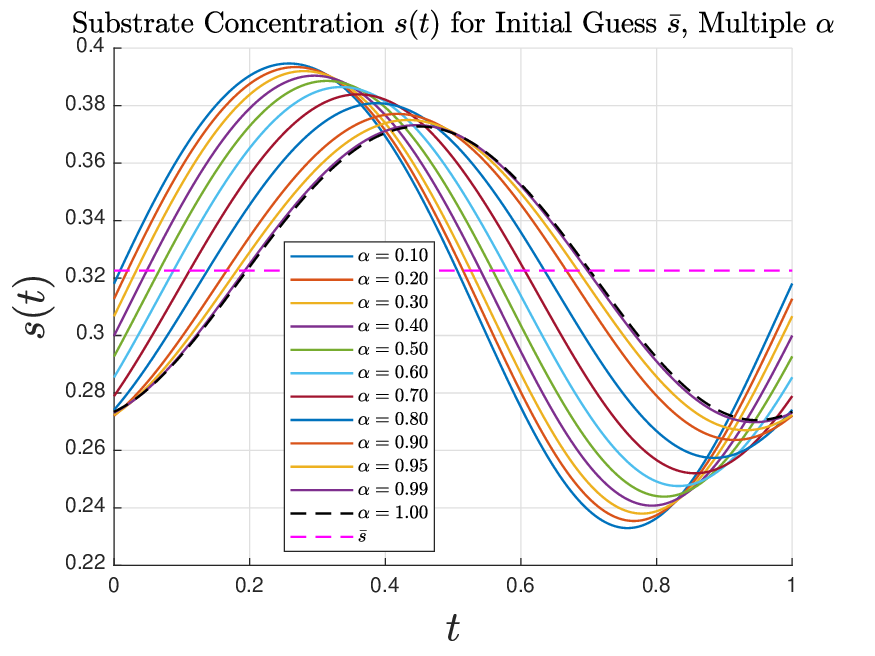}
\caption{Time evolution of substrate concentration $s(t)$ for various fractional orders $\alpha$ in the interval $[0.1,1]$, using the steady-state value $\bar{s}$ as the initial guess for all spatial nodes. Each colored solid curve represents a different fractional order $\alpha$. The dashed black line corresponds to the classical integer-order case $\alpha=1$. The dashed magenta line indicates the steady-state substrate concentration $\bar{s}$.}
\label{fig:Ex14}
\end{figure}

Figure \ref{fig:Ex15} illustrates the time evolution of the substrate concentration $s(t)$ under the configuration set $\mathcal{D}$, but for various memory lengths $L$, and using the steady-state value $\bar{s}$ as the initial guess at all spatial nodes. At the beginning of the simulation, the distinction between the trajectories is obvious---smaller values of $L$ display greater deviations, reflecting the stronger influence of recent history on the system's dynamics. As $L$ increases, the trajectories gradually converge, particularly for $L = 1,\,3,\,5$, indicating that the impact of the memory window saturates and the system's behavior becomes less sensitive to further increases in $L$.

Mathematically, this convergence arises because the CFDS increasingly incorporate a longer segment of the system's history into its present evolution and smooth out short-term fluctuations. Physically, this implies that, for large memory lengths, the system's dynamics are governed by extended historical behavior, resulting in more uniform trajectories. Biologically, it corresponds to a scenario where microorganisms in the chemostat respond not just to recent substrate concentrations, but to a longer history of environmental conditions, thereby producing more stable and predictable substrate dynamics.

This behavior emphasizes the importance of memory effects in the FOCM and their role in capturing the adaptive and history-dependent characteristics of biological systems. The greater variability in $s(t)$ observed for shorter memory lengths ($L = 0.1,\,0.3,\,0.5$) reflects the system’s sensitivity to recent changes in the dilution rate, which drives microbial consumption of pollutants. This has practical biological relevance in wastewater treatment systems subject to rapid influent fluctuations, where microbial activity is more immediately influenced by recent nutrient availability.

The convergence of substrate trajectories for larger $L$ is consistent with the numerical convergence depicted in the figure, suggesting that the FDE's unique periodic solution dominates as more historical data are integrated. This leads to stabilized microbial consumption, with $x(t)$ remaining strictly positive, indicating sustained microbial activity. Biologically, this captures treatment scenarios where pollutant degradation is influenced by a broader history of substrate availability, promoting more reliable and efficient outcomes. Such behavior is advantageous in municipal wastewater systems with relatively stable influent, where predictable microbial responses improve treatment performance. The bounded oscillations of $s(t)$ around $\bar{s}$ ensure pollutant concentrations remain controlled, and the condition $x(t) > 0$ affirms ongoing biomass growth. The finite memory window of the CFDS effectively models realistic microbial adaptation, unlike classical derivatives, which lack this historical sensitivity.

From an operational standpoint, tuning the memory length $L$ in the FDE offers a practical mechanism for optimizing bioreactor performance. Shorter values of $L$ can improve responsiveness under dynamic influent conditions, while longer values promote stability, supporting robust pollutant degradation in diverse wastewater treatment environments.

\begin{figure}[H]
\centering
\includegraphics[scale=0.5]{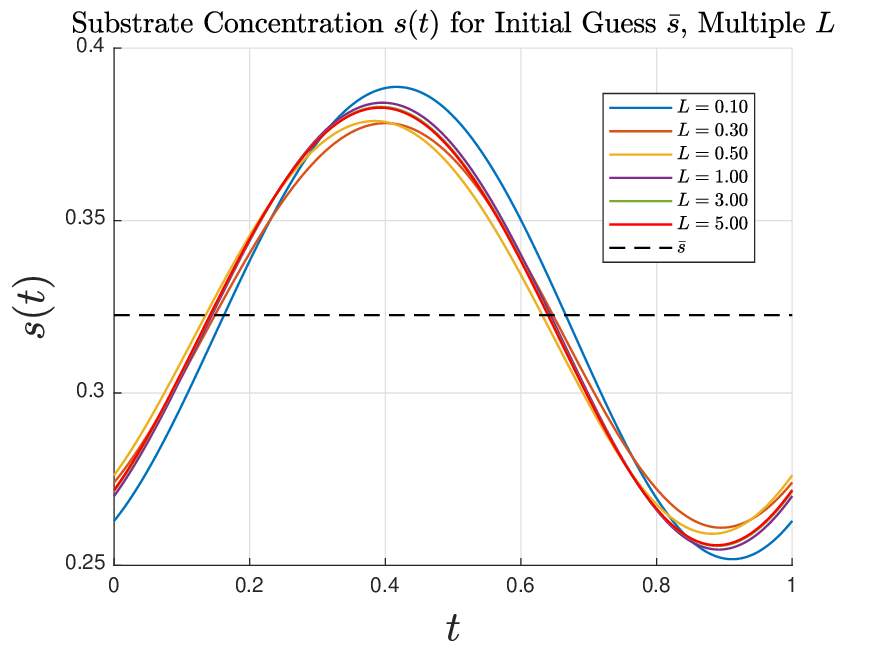}
\caption{Time evolution of substrate concentration $s(t)$ for various memory lengths $L \in \{0.1,\,0.3,\,0.5,\,1,\,3,\,5\}$, using the steady-state value $\bar{s}$ as the initial guess for all spatial nodes. Each colored solid curve represents a different value of $L$. The dashed black line indicates the steady-state substrate concentration $\bar{s}$.}
\label{fig:Ex15}
\end{figure}

Figure \ref{fig:Ex16} depicts the substrate concentration $s(t)$ using the initial guess $\bar{s}$ and the configuration set $\mathcal{D}$, but for multiple values of $\vartheta$. The figure illustrates the system's response under varying conditions. In particular, the trajectories exhibit a clear oscillatory pattern around the central value $\bar{s}$, with both the amplitude and phase influenced by the parameter $\vartheta$. As $\vartheta$ increases, the oscillations become larger, reflecting a stronger influence of memory effects on the system dynamics. However, for sufficiently large $\vartheta$, the curves converge to a unique periodic trajectory, indicating that the impact of memory effects saturates and the system's behavior becomes less sensitive to further increases in $\vartheta$. This behavior is consistent with the theoretical interpretation of $\vartheta$ as a characteristic time constant, which modulates the weighting of past states in the CFDS. Biologically, this implies that larger values of $\vartheta$ improve the sensitivity of microbial activity to historical substrate variations, potentially affecting the efficiency of pollutant degradation in the chemostat.

The figure highlights the critical role of $\vartheta$ in shaping both memory effects and system variability, offering valuable insights for optimizing periodic wastewater treatment processes. The increased amplitude of substrate oscillations for larger $\vartheta$ values indicates that the FDE amplifies the influence of historical substrate conditions on current dynamics. Despite these variations, the convergence to a unique periodic solution confirms that the substrate–dilution dynamics remain stable, with $x(t)$ consistently indicating robust microbial activity.

From a biological perspective, higher values of $\vartheta$ correspond to bioreactors with longer characteristic time scales (e.g., hydraulic retention times), where microbial consumption becomes more responsive to earlier substrate levels. This is particularly relevant in treating complex pollutants that require prolonged microbial exposure. The oscillatory behavior around $\bar{s} \approx 0.323$ ensures that substrate concentrations remain within biologically meaningful bounds, thereby preventing pollutant accumulation. The consistent convergence to a periodic solution, regardless of initial conditions, demonstrates the FDE's capacity to model stable substrate dynamics governed by the periodic dilution rate $D(t)$.

Practically, $\vartheta$ serves as a tunable parameter for controlling substrate dynamics in the FOCM. By adjusting $\vartheta$, one can regulate oscillatory behavior to optimize microbial degradation of pollutants under different wastewater treatment regimes.

\begin{figure}[H]
\centering
\includegraphics[scale=0.5]{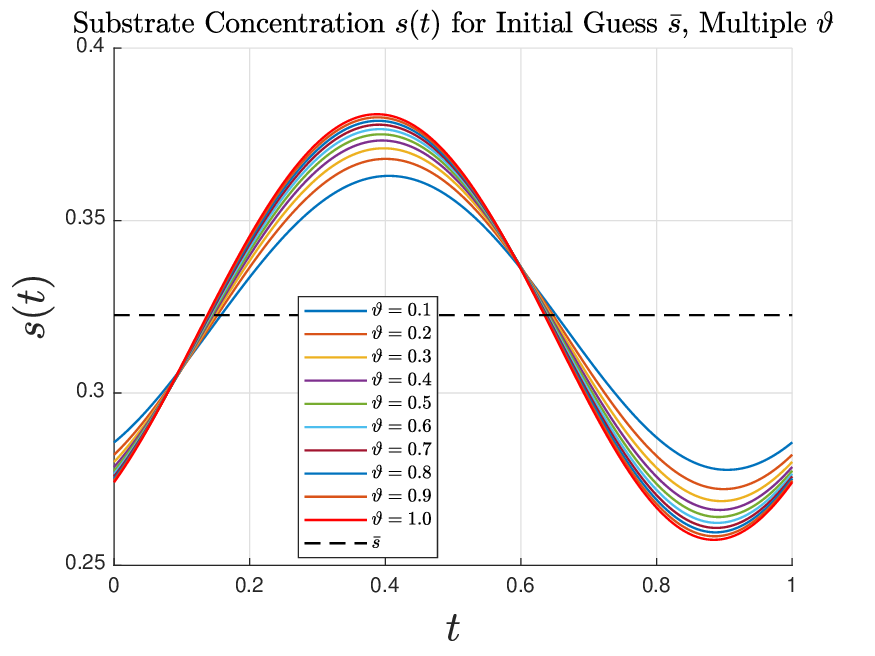}
\caption{Time evolution of substrate concentration $s(t)$ for various values of the characteristic time constant $\vartheta \in \{0.1,\,0.2,\,\ldots,\,1.0\}$, using the steady-state value $\bar{s}$ as the initial guess for all spatial nodes. Each colored solid curve represents a different value of $\vartheta$. The dashed black line indicates the steady-state substrate concentration $\bar{s}$.}
\label{fig:Ex16}
\end{figure}

Figure \ref{fig:Ex13} displays the substrate concentration profiles computed for 100 random initial guesses using the configuration dataset $\mathcal{D}$, but with $\mu_{\max} = 0.25$ and $K = 2$. Under this setting, Theorem \ref{thm:uniqueness}, Condition (i) guarantees that the trivial solution $s(t) = s_{\text{in}}$, which is the maximum substrate concentration with zero biomass concentration $x(t) = 0$, is the unique solution to the FOCM. This theoretical prediction is fully consistent with the figure, where all computed trajectories converge to the horizontal line $s(t) = s_{\text{in}}$, regardless of the initial guess. This is expected when no biomass is present to consume the substrate, preventing any pollutant degradation. The increased saturation constant $K = 2$ in the Contois model strengthens biomass inhibition in the FDE, which limits microbial consumption even at low substrate levels, contributing to washout. This highlights the importance of balancing system parameters to support microbial activity.

\begin{figure}[H]
\centering
\includegraphics[scale=0.5]{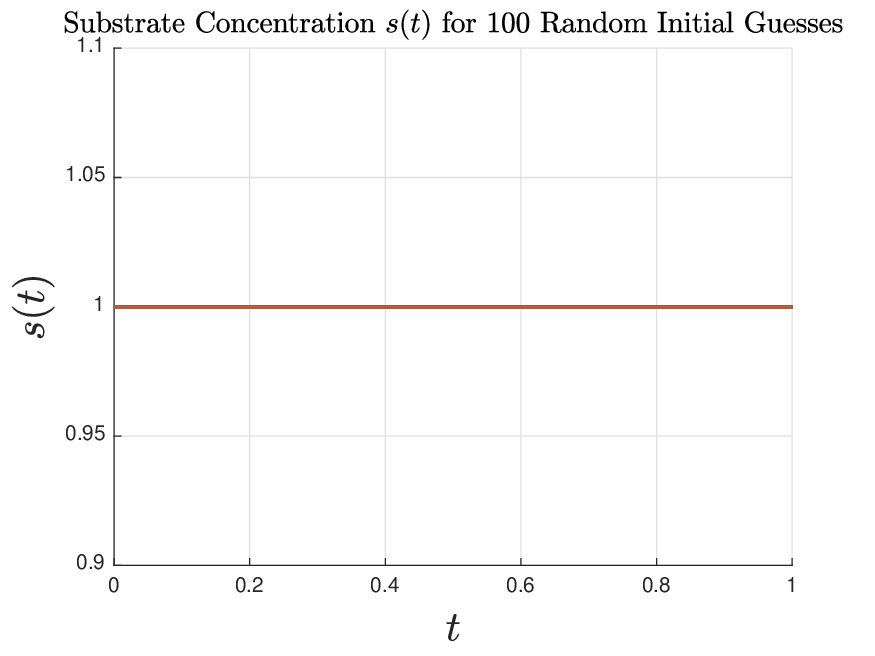}
\caption{Approximate substrate concentration profiles $s(t)$ over one period $T$ obtained by numerically solving the FOCM with Contois growth function, using $\mu_{\max} = 0.25$ and $K = 2$, for $100$ different random initial guesses. Each initial guess vector was generated by sampling each element independently and uniformly at random from the interval $[0, s_{\text{in}}]$. All curves are identical straight lines at $s(t) = s_{\text{in}}$, reflecting convergence from distinct initial substrates to the trivial solution.}
\label{fig:Ex13}
\end{figure}

Consider next the same configuration dataset $\mathcal{D}$ of the FOCM but with the bang-bang dilution rate: 
\begin{equation}\label{eq:DiscDil1}
D(t) = 
\begin{cases}
D_{\max}, & \text{if } t \in [0.25,\, 0.75), \\
D_{\min}, & \text{otherwise},
\end{cases}
\end{equation}
where $D_{\min} = 0.5$ and $D_{\max} = 1.5$, so that $D(t) \in [0.5, 1.5]$. Since $\bar{D} = 1$, the equilibrium substrate concentration remains $\bar{s} \approx 0.323$.

The numerical results presented in Figure \ref{fig:Ex17} demonstrate the robustness of the FG-PS method in solving the FOCS governed by the one-dimensional FDE with PBCs. As shown in Figure \ref{fig:Ex17}(a), the substrate concentration $s(t)$ exhibits a continuous, $T$-periodic oscillation around the equilibrium value $\bar{s} \approx 0.323$. This behavior is consistent with the theoretical predictions of Theorem \ref{thm:existence}, which guarantees the existence of a non-trivial, $T$-periodic Carath\'{e}odory solution. The absolute residuals, plotted in Figure~\ref{fig:Ex17}(d), are on the order of $10^{-14}$, confirming the spectral convergence of the FG-PS method. The bang-bang dilution rate $D(t)$, depicted in Figure \ref{fig:Ex17}(b), alternates sharply between $D_{\min} = 0.5$ and $D_{\max} = 1.5$ within each period. These discontinuities introduce corners in the substrate trajectory $s(t)$, which nevertheless remains continuous in accordance with the Carath\'{e}odory framework. The sliding memory window $[t - L, t]$, which integrates the system's past states, plays a critical role in smoothing these corners and maintaining the continuity of $s(t)$.

The initial condition $s(0) \approx 0.272 < s^* = 0.5$ and the bound $D(t) < \nu(s^*) = 1.55$ together support the existence of a unique and stable non-trivial periodic solution, as established by Theorem \ref{thm:uniqueness}. The application of the CFDS scheme ensures the preservation of periodicity and provides advantages over classical derivatives in terms of mathematical formulation.

The $T$-periodic oscillation of the substrate concentration around $\bar{s} \approx 0.323$, driven by the bang-bang dilution rate $D(t)$, reflects the microbial system's adaptability to abrupt fluctuations in nutrient input. Biologically, this illustrates how microorganisms can maintain functional stability in environments governed by on-off nutrient dosing strategies. The bang-bang dilution rate mimics practical bioreactor control mechanisms, where nutrient input is discretely toggled rather than continuously varied.

Importantly, the continuity of the substrate profile despite discontinuities in $D(t)$ demonstrates the smoothing effect of the sliding memory window in the FDE, which stabilizes substrate dynamics and supports consistent microbial consumption. From a treatment perspective, the stable periodic solution under bang-bang dilution rate validates the efficacy of on-off feeding strategies in optimizing pollutant degradation. The FDE's focus on substrate-dilution dynamics, with the microbial population $x(t)$ derived post hoc, provides a robust and flexible modeling framework for the design and analysis of efficient wastewater treatment systems.

\begin{figure}[H]
\centering
\includegraphics[scale=0.65]{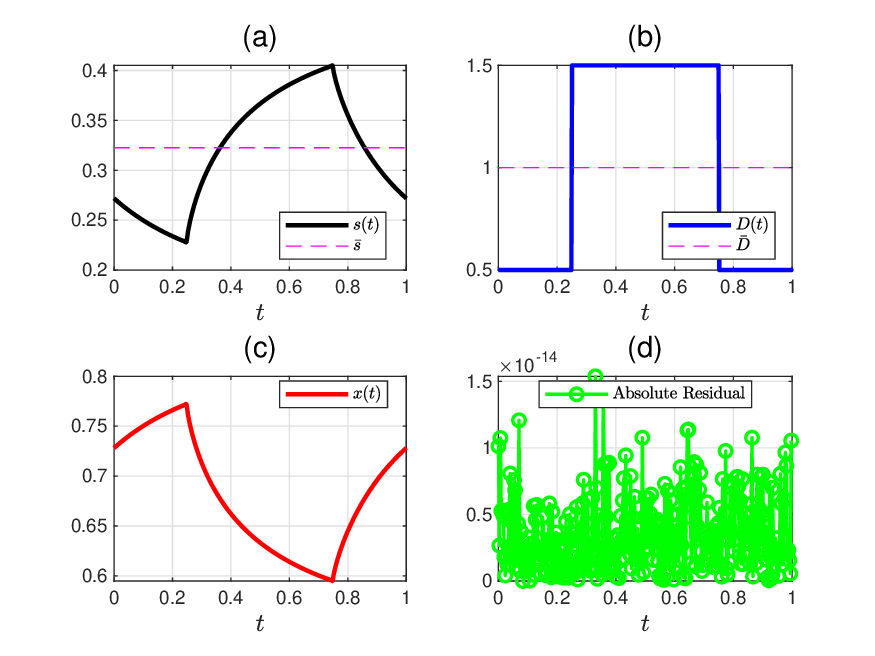}
\caption{Simulation results of the FOCM with Contois kinetics using the FG-PS method using the bang-bang dilution rate \eqref{eq:DiscDil1} and under the remaining original parameter dataset: (a) Substrate concentration $s(t)$ vs. steady state $\bar{s}$, (b) Bang-bang dilution rate $D(t)$ vs. average $\bar{D}$, (c) Biomass concentration $x(t)$, (d) Absolute residuals at collocation nodes.}
\label{fig:Ex17}
\end{figure}

\section{Conclusion}
\label{sec:Conc}
In this work, we developed a rigorous mathematical framework for modeling periodic biological water treatment using a novel FOCS. The model incorporates the CFDS and PBCs, capturing essential memory effects and time-periodic behaviors often overlooked in traditional integer-order models. Using the Contois growth function, we reduced the original two-dimensional FDE system to a one-dimensional FDE, enabling detailed mathematical analysis. We proved the well-posedness of the model by establishing the existence, uniqueness, positivity, and boundedness of non-trivial periodic Carath\'{e}odory solutions. These results provide a mathematically sound basis for understanding the dynamic behavior of microbial systems in periodically operated chemostats under realistic biological constraints. Furthermore, we rigorously characterized the steady states under constant dilution rates and derived precise conditions for the existence and uniqueness of non-trivial periodic solutions. 

This study is the first to combine the CFDS, Contois kinetics, and periodic operation in a chemostat setting with formal mathematical guarantees. The developed theory demonstrates that sliding memory effects and periodic forcing can be harnessed to produce stable, biologically feasible operating regimes, which is critical for the design of advanced water treatment systems. 

Looking forward, this framework opens promising directions for future research, such as the inclusion of time delays, stochastic disturbances, or optimal periodic dilution rate strategies tailored to sustainability goals. The results contribute to bridging the gap between mathematical theory and practical engineering challenges in environmental biotechnology, particularly in achieving Sustainable Development Goal 6 (Clean Water and Sanitation).

\section*{Declarations}
This research aligns with global sustainability goals, particularly SDG 6 (Clean Water and Sanitation), by advancing mathematical tools for efficient and scalable water treatment solutions. The theoretical nature of this work paves the way for future experimental validation and real-world implementation.

\subsection*{Competing Interests}
The author declares that they have no competing interests.

\subsection*{Availability of Supporting Data}
Not applicable. This study is theoretical in nature and does not involve experimental data or datasets. All mathematical derivations and proofs are contained within the manuscript.

\subsection*{Ethical Approval and Consent}
Not applicable.

\subsection*{Funding}
This research received no specific grant from any funding agency in the public, commercial, or not-for-profit sectors.


\bibliographystyle{plainnat}
\bibliography{Bib}

\end{document}